\documentclass[11pt]{amsart}

\usepackage{enumerate}\usepackage{latexsym}
\usepackage{amssymb}
\usepackage{amsmath}
\usepackage{color}

\usepackage
{hyperref}
\hypersetup{%
  colorlinks,
  bookmarksopen,
  bookmarksnumbered,
  citecolor=blue,
  linkcolor=black,
  pdfstartview=FitH,
}

\newtheorem{theorem}{Theorem}[section]
\newtheorem{lemma}[theorem]{Lemma}
\newtheorem{proposition}[theorem]{Proposition}
\newtheorem{corollary}[theorem]{Corollary}

\theoremstyle{definition}
\newtheorem{example}[theorem]{Example}
\newtheorem{definition}[theorem]{Definition}
\newtheorem{remark}[theorem]{Remark}

\DeclareMathOperator{\id}{id}
\DeclareMathOperator{\tr}{tr}

\newcommand{\cl}[1]{\mathcal{#1}}
\newcommand{\bb}[1]{\mathbb{#1}}
\newcommand{\A}{\cl A}
\newcommand{\B}{\cl B}
\newcommand{\D}{\cl D}
\renewcommand{\P}{\cl P}

\newcommand{\bC}{\bb C}

\newcommand{\bN}{\bb N}

\DeclareMathOperator{\Ad}{Ad}
\DeclareMathOperator{\Ker}{Ker}
\DeclareMathOperator{\spn}{span}

\newcommand{\nor}[1]{\left\Vert #1\right\Vert}
\let\norm\nor
\newcommand{\la}{\langle}
\newcommand{\ra}{\rangle}
\newcommand{\vphi}{\varphi}
\newcommand{\ep}{\varepsilon}
\newcommand{\om}{\omega}
\newcommand{\Om}{\Omega}
\newcommand{\BH}{\mathcal{B}(H)}

\newcommand{\N}{\mathbb{N}}
\newcommand{\Z}{\mathbb{Z}}
\newcommand{\R}{\mathbb{R}}
\let\C\outer\renewcommand{\C}{\mathbb{C}}
\newcommand{\ten}{\otimes}
\newcommand{\oten}{\overline{\otimes}}

\newcommand{\TO}{\ensuremath{\mathrm{II}_1}}

\newcommand{\qtext}[1]{\quad\text{#1}\quad}
\newcommand{\qand}{\qtext{and}}

\let\phi\varphi
\let\nph\varphi
\let\epsilon\varepsilon

\newcommand{\revduality}[2]{{
    \left\langle#1,#2\right\rangle}} 
\newcommand{\duality}[2]{\revduality{#2}{#1}} 

\newcommand{\ip}[2]{{
    \left(#1,#2\right)}}

\newcommand{\ext}{\rm ext}
\DeclareMathOperator{\cbm}{CB}
\DeclareMathOperator{\cp}{CP}
\DeclareMathOperator{\ucp}{UCP}
\DeclareMathOperator{\locc}{LOCC}
\let\LOCC\locc
\let\UCP\ucp
\let\CP\cp

\let\mpar\marginpar
\def\marginpar#1{\mpar{#1}\ignorespaces}

\begin{document}

\title[State Convertibility in the von Neumann Algebra Framework]{State Convertibility in the von Neumann Algebra Framework}

\author[J. Crann]{Jason Crann}
\address{School of Mathematics \& Statistics, Carleton University, Ottawa, ON, Canada H1S 5B6}
\email{jasoncrann@cunet.carleton.ca}

\author[D. W. Kribs]{David W. Kribs}
\address{Department of Mathematics \& Statistics, University of Guelph, Guelph, ON, Canada N1G 2W1 \& Institute for Quantum Computing, University of Waterloo, Waterloo, ON, Canada N2L 3G1}
\email{dkribs@uoguelph.ca}

\author[R. H. Levene]{Rupert H. Levene}
\address{School of Mathematics and Statistics, University College Dublin, Belfield, Dublin 4, Ireland}
\email{rupert.levene@ucd.ie}

\author[I. G. Todorov]{Ivan G. Todorov}
\address{Mathematical Sciences Research Centre, Queen's University Belfast, Belfast BT7 1NN, United Kingdom,
  and School of Mathematical Sciences, Nankai University, 300071 Tianjin, China}
\email{i.todorov@qub.ac.uk}

\date{April 2020}



\begin{abstract}
We establish a generalisation of the fundamental state
convertibility theorem in quantum information to the context of
bipartite quantum systems modelled by commuting semi-finite von
Neumann algebras. Namely, we establish a generalisation to this
setting of Nielsen's theorem on the convertibility of quantum
states under local operations and classical communication (LOCC)
schemes. Along the way, we introduce an appropriate generalisation
of LOCC operations and connect the resulting notion of approximate
convertibility to the theory of singular numbers and majorisation
in von Neumann algebras. As an application of our result in the
setting of $\TO$-factors, we show that the entropy of the singular
value distribution relative to the unique tracial state is an
entanglement monotone in the sense of Vidal, thus yielding a new
way to quantify entanglement in that context. Building on previous
work in the infinite-dimensional setting, we show that trace
vectors play the role of maximally entangled states for general
$\TO$-factors. Examples are drawn from infinite spin chains,
quasi-free representations of the CAR, and discretised versions of
the CCR.
\end{abstract}

\maketitle

\section{Introduction}\label{s_bc}

Quantum entanglement is a central notion in quantum information theory and a key resource in the applications that are driving efforts to develop quantum technologies. While there is a growing depth of understanding of the concept and its many potential uses, the theory of quantum entanglement remains an active, challenging, and fundamental area of investigation in quantum information theory with, of particular note here, relatively little progress having been made in the general infinite-dimensional and von Neumann algebra settings.

The mathematical theory that provides the foundation for these efforts rests in many ways on an understanding of how entanglement between quantum states can be transformed through various tasks and processes. A very natural and central question is to ask whether a given state, entangled between multiple parties, can be transformed by certain restricted classes of quantum operations to other types of entangled states, with restrictions determined by theoretical or physical limitations. This is a basic question that is relevant, for instance, in the development of any quantum communication scheme, realisations of quantum algorithms, physical implementations of quantum networks, etc.

As the simplest starting point in this subject, consider the scenario in which two parties $A$ and $B$ each have the ability to implement all local quantum operations, described mathematically by completely positive and trace-preserving maps on
the algebras of bounded linear operators on
their respective system Hilbert spaces $H_A$, $H_B$, but such that the parties are limited in that they can only communicate with each other using classical communication. An initial core problem then is to start with an entangled state
$\psi \in H_A \otimes H_B$ shared by the parties, and to determine what are the possible entangled states that $\psi$ can be converted to through local operations and classical communication (LOCC) between them.

This question has a neat matrix theoretic solution in the finite-dimensional case, known as Nielsen's Theorem \cite{nielsen}.
For every pure state $\psi \in H_A \otimes H_B$, let $\rho_\psi = \tr_B (\psi\psi^*)$ be the (in general, mixed) state on $H_A$ found by applying the partial trace map
over $H_B$ to the projection $\psi\psi^*$ with range the one-dimensional space of scalar multiples of $\psi$.
The eigenvalues of $\rho_\psi$ (including multiplicities) form a probability distribution and can be arranged in non-increasing order, thus giving rise to a real vector $\lambda_\psi$.
Nielsen's Theorem states that $\psi$ can be converted into another state $\phi$ by LOCC between $A$ and $B$ if and only if
$\lambda_\psi$ is majorised by $\lambda_\phi$, that is, all partial sums of values from $\lambda_\psi$ (respecting the non-increasing indexing) are bounded above by the corresponding partial sums of values from $\lambda_\phi$. Recall that $\rho_\psi$ is a pure (rank one) state if and only if ${\psi}$ is separable. Obviously if we start with a separable state ${\psi}$, then we can only transform it to another separable state via LOCC, and this case is easily captured by the theorem with $\lambda_\psi=(1,0,\ldots ,0)$. On the other hand, given any separable state ${\phi}$, any arbitrary state ${\psi}$ can be transformed to it via LOCC, in particular by making use of local depolarising maps on the individual systems.
The power of Nielsen's theorem lies in the fact that it gives a matrix and spectral theoretic description of which entangled states are attainable through LOCC when we start with a given entangled state. The theorem has far-reaching implications and applications
throughout finite-dimensional quantum information theory; indeed, it is one of the most important and widely used results in the entire field.

While the primary focus of research in quantum information has been on challenges and applications in the finite-dimensional and qubit setting for more than two decades now, ultimately general quantum mechanics is an infinite-dimensional theory that is rooted in the theory of von Neumann algebras \cite{vN}. Thus, one can reasonably expect that continued long-term progress in quantum information theory and its connections within theoretical physics will depend at least partly on the successful extension of central results in the field to the infinite-dimensional and general von Neumann algebra settings, with new peculiarities and connections uncovered along the way.  This is clearly a desirable goal, and there has been a recent reemergence of activity in this direction, including quantum error correction and privacy (e.g.~\cite{bkk2,cklt}), entropy theory (e.g.~\cite{bfs,hiaif1,hiaif2,longo,lx}), Bell inequalities (e.g.~\cite{jungep}), entanglement in quantum field theory (e.g.~\cite{hs} and the references therein), and most notably, connections with, and recent refutation of, Connes' embedding conjecture (e.g.~\cite{dpp,musath,jietal,jungeetal,slofstra}).

In this paper, we make new progress in this direction, establishing, as our main result, a generalisation of Nielsen's Theorem
to the context of von Neumann algebras, and more specifically for bipartite quantum systems modelled by commuting semi-finite von Neumann algebras,
say $\cl A$ and $\cl B$ \cite{haagerup2,t1,t2}.
We note that the case where $\cl A$ and $\cl B$ are separably acting factors of type I was considered in \cite{obnm}.
While some parts of the theory extend in a somewhat straightforward way, there are, as one would expect, significant technical challenges to overcome,
well beyond the generalisation provided in \cite{obnm}.
En route, we introduce an appropriate generalisation of LOCC operations \cite{clmow} to our context.
The setting for our version of Nielsen's Theorem is provided by
the theory of singular numbers \cite{fk} and majorisation \cite{h,h2} in von Neumann algebras.
We build our analysis on key aspects of operator algebra theory, such as the standard form of a von Neumann algebra, the theory of completely bounded maps, the Haagerup tensor product, and dilation theory \cite{fk,t2,blecher_smith,haagerup,haagerup2,h,ksw}.

We include a number of examples and applications of our results.
In particular, we show that the entropy of the singular value distribution relative to the unique tracial state of a type $\TO$-factor
is an entanglement monotone in the sense of Vidal \cite{v,op}, thus yielding a new way to quantify entanglement in that context. Building on previous work in the infinite-dimensional setting, we also show that trace vectors play the role of maximally entangled states for general $\TO$-factors. Examples are drawn from quantum measurement theory \cite{vN}, infinite spin chains \cite{keylsw,kmsw}, quasi-free representations of the CAR \cite{bcs,dg,dmm}, and discretised versions of the CCR \cite{arv,bkk2,Faddeev}.

This paper is organised as follows. The next section includes requisite preliminaries, focussed mainly on von Neumann algebra theory. In Section~\ref{s_locc},
we introduce the general notion of LOCC maps and derive some basic properties. In Section~\ref{s:conv}, we investigate approximate convertibility of states by LOCC maps, showing that we can restrict attention to one-way convertibility. This is used in Section~\ref{s:main} to establish our main result on state convertibility and majorisation; this section also includes several supporting technical results. Section~\ref{s_ttf} includes the aforementioned applications and examples. Other illustrative examples are presented throughout the paper. We finish with a brief outlook discussion on our results and the subject.

\section{Preliminaries}

Let $\cl A$ be a von Neumann algebra.
We denote by $\cl A_*$ its predual; thus, the elements of $\cl A_*$ are normal (that is, weak* continuous) linear functionals on $\cl A$.
If $\cl B$ is another von Neumann algebra with $\cl A\subseteq \cl B$,
a map $\Phi : \cl B\to \cl B$ is called an $\cl A$-bimodule map if $\Phi(axb) = a\Phi(x)b$, for all $a,b\in \cl A$ and all $x\in \cl B$.
We denote by
$\cp_{\cl A}^\sigma(\cl B)$ the cone of all normal completely positive $\cl A$-bimodule maps on $\cl B$.
We let $\ucp_{\cl A}^\sigma(\cl B)$ stand for the convex subset of all unital maps in $\cp^\sigma_{\cl A}(\cl B)$.
We set $\cp^\sigma(\cl B) = \cp_{\bb{C}I}^\sigma(\cl B)$ and $\ucp^\sigma(\cl B) = \ucp_{\bb{C}I}^\sigma(\cl B)$.
We call the elements of $\ucp^\sigma(\cl B)$ \emph{quantum channels} on $\cl B$.
We refer the reader to \cite{paulsen_book} for basics of the theory of completely positive and completely bounded maps and some standard notation.

We denote by $S(\cl A)$ the convex set of all normal states of $\cl A$,
and by $S_{\ext}(\cl A)$ the set of all pure normal states of $\cl A$.
If $\omega\in \cl A_*$ and $a\in \cl M$, define $a\omega, \omega a\in \cl A_*$ by
$(a\omega)(x) = \omega(xa)$ and $(\omega a)(x) = \omega(ax)$.
We sometimes use the duality pairing notation $\duality{\omega}{x}:=\omega(x)$ with angled brackets (in contrast, we use rounded parentheses for inner products in Hilbert spaces); in this notation, the bimodule action just described may
be written as $\duality{a\omega b}{x}=\duality{\omega}{bxa}$, $a,b\in \cl A$.

As usual, we let $\cl B(H)$ be the $C^*$-algebra of all bounded operators on a Hilbert space $H$, and set
$M_n=\B(\bC^n)$, the algebra of $n\times n$ matrices with complex entries. We assume throughout this paper that all Hilbert spaces under consideration are separable, sometimes mentioning this explicitly for emphasis.
For an element $a\in \cl B(H)$, we write $\Ad(a)$ for the map
given by $\Ad(a)(x) = axa^*$, where $a^*$ denotes the adjoint of~$a$; clearly, $\Ad(a) \in \cp^\sigma(\cl B(H))$.
If $\nph,\psi\in H$, we let $\nph\psi^*$ be the rank one operator on $H$ given by $(\nph\psi^*)(\xi) = (\xi,\psi)\nph$ (note that our inner products are linear in the first argument). Let $\cl T(H)$ be the space of all trace class operators in $\cl B(H)$ and ${\rm tr} : \cl T(H)\to \bb{C}$ be the trace.
We have a canonical
identification $\cl T(H) \equiv \cl B(H)_*$, given by letting $\duality{S}{T} = {\rm tr}(TS)$, $T\in \cl B(H)$,
$S\in \cl T(H)$.
We denote by $H_1$ the set of unit vectors in a Hilbert space~$H$.
If $\phi\in H$, we write $\omega_\phi\in \BH_*$ for the (positive, normal) functional given by
$\omega_\phi(x) = \ip{x\phi}{\phi}$.

Throughout the paper, we will let $\cl A\subseteq \BH$ be a von Neumann algebra, for some (separable) Hilbert space~$H$, with unit $1$, projection lattice $\mathcal P(\cl A)$ and
positive cone $\cl A^+$.
We will mainly be interested in the case when $\cl A$ is semi-finite, equipped with a normal semi-finite faithful trace $\tau$.
Let $\tilde{\cl A}$ be the *-algebra of all $\tau$-measurable operators \cite{fk}, that is,
the set of all densely defined closed operators $T : \cl D(T)\to H$, (where
$\cl D(T) \subseteq H$ is the domain of~$T$), affiliated with $\cl A$, with the property that for every $\epsilon > 0$
there exists a projection $e\in \cl A$ such that $\tau(1-e) \leq \epsilon$ and $eH \subseteq \cl D(T)$.
For $p\geq 1$, let
$$\cl A_p = \{a\in \cl A : \tau(|a|^p) < \infty\},$$
and let
$L^p(\cl A,\tau)$ be the completion of $\cl A_p$
with respect to the norm $\|\cdot\|_p$, given by
$\|a\|_p = \tau(|a|^p)^{1/p}$, $a\in \cl A_p$.
Set $L^{\infty}(\cl A,\tau) = \cl A$.
We will extensively use the fact that the elements of the space $L^p(\cl A,\tau)$ can be canonically identified with operators in $\tilde{\cl A}$ (see \cite{fk}).

Note that $L^2(\cl A,\tau)$ is a Hilbert space, which is separable since $\cl A$ is separably acting, with inner product given by
$$\ip{a}{b} = \tau(b^*a), \ \ \ a,b\in \cl A_2.$$
In fact,  $L^2(\cl A,\tau)$ is
the Hilbert space arising from the GNS construction applied to $\tau$.
The associated (normal, faithful) *-representation $\pi_\tau : \cl A \to \cl B(H)$ is given by
$$\pi_{\tau}(a)b = ab, \ \ \ b\in \cl A_2, \ a\in \cl A.$$
We will suppress the use of the notation $\pi_{\tau}$; in this way, we will consider $\cl A$ as a von Neumann subalgebra of $\cl B(L^2(\cl A,\tau))$.
We then have that $\cl A$ is in its standard form \cite{haagerup2}; we also say that $\A$ is standardly represented on~$L^2(\cl A,\tau)$. 
Working in this standard representation, let $\cl A'$ be the commutant of $\cl A$, and let
$J : L^2(\cl{A},\tau) \to L^2(\cl{A},\tau)$ be the associated conjugate linear isometry
with the property that $\cl{A}' = J\cl{A}J$.
Note that $J$ is the (unique) extension of the
adjoint map $a\mapsto a^*$ on $\cl A_2$, 
and for $\psi\in L^2(\cl A,\tau)$, we have that $\psi^* = J\psi$, where the
left hand side in the latter identity is the adjoint of the linear operator $\psi$ (by which we mean the element of
$\tilde{\cl A}$ canonically identified with~$\psi$). 
For $\xi,\eta\in L^2(\cl A,\tau)$, we have
\begin{equation}\label{eq_jaj}
(J\xi,J\eta) = (\xi^*,\eta^*) = \tau(\eta\xi^*) = \tau(\xi^*\eta) = (\eta,\xi).
\end{equation}
The map
$\pi'_\tau : \cl{A}\rightarrow \cl{B}(L^2(\cl{A},\tau))$, given by
$$\pi'_\tau(a)b = ba, \ \ \ b\in \cl A_2, \ a\in \cl A,$$
is a faithful anti-*-homomorphism,
satisfying $\pi'_\tau(a) = Ja^*J$, $a\in \cl{A}$ (see \cite[Theorem V.2.22]{t1}).
Let $R : \cl A'\rightarrow \cl A$ be the anti-*-isomorphism given by
$$R(a') = Ja'^*J, \ \ \ a'\in\cl A'.$$

We note that $L^1(\cl A,\tau)$ can be identified in a canonical way with the predual of $\cl A$.
In fact, if $\omega\in \cl A_*$ then there exists a unique $\rho_{\omega}\in L^1(\cl A,\tau)$ such that
$$\omega(a) = \tau(\rho_{\omega}a), \ \ \ a\in \cl A.$$
Here, and in the sequel, we use the fact that $L^1(\cl A,\tau)$ is an $\cl A$-bimodule, that is,
given $\rho\in L^1(\cl A,\tau)$ and $a,b\in \cl A$, we have that $a\rho b$ is a well-defined $\tau$-measurable operator and belongs to $L^1(\cl A,\tau)$.
Note that if $\omega\in \cl A_*^+$, then $\rho_{\omega}$ is a positive (in general unbounded) operator,  which we call the
\emph{density operator} of $\omega$.

If $x\in \tilde{\cl A}$ and $t > 0$, the \emph{$t$-th singular value} $\mu_t(x)$
of $x$ is defined by letting
\begin{equation*}
\mu_t(x) = \mu_t(x;\A,\tau)= \inf\{\|xp\| : p\in\mathcal P(\cl A),\, \tau(1-p)\leq t\}.
\end{equation*}
The \emph{singular value function of~$x$},
namely $\mu(x)\colon (0,\infty)\to (0,\infty)$, $t\mapsto \mu_t(x)$, is decreasing and continuous from the right \cite[Lemma 2.5]{fk}.
If $x,y\in \tilde{\cl A}^+$,
we say that $x$ is \emph{majorised} by $y$ if
$$\int_0^s \mu_t(x)\, dt \leq  \int_0^s \mu_t(y)\, dt, \ \ \ 0< s \leq \infty;$$
we write $x\prec y$ to designate the fact that $x$ is majorised by $y$ and $\tau(x) = \tau(y)$.
We refer to \cite{h} for extensive details on majorisation of elements of $\tilde{\cl A}^+$ and to
\cite{fk} for background on the theory of singular values.

Let $\tau' : \cl A'\to \bb{C}$ be the functional given by $\tau'(a') = \tau(R(a'))$, $a'\in \cl A'$.
Then $\tau'$ is a normal faithful
semi-finite trace on $\cl A'$.
Since $\cl A'\subseteq \cl B(L^2(\cl A,\tau))$, the elements of $L^1(\cl A',\tau')$ can be identified with
linear densely defined operators on $L^2(\cl A,\tau)$.
Given a normal functional $\omega'$ on $\cl A'$, there exists, by the preceding discussion, a (unique) element $\rho'_{\omega'}\in L^1(\cl A',\tau')$
such that $\omega'(b) = \tau'(\rho'_{\omega'}b)$, $b\in \cl A'$.
The constructions described above can be
performed relative to the pair $(\cl A',\tau')$;
in particular, one may define the corresponding singular values
$\mu'_t(x'):=\mu_t(x';\A',\tau')$ associated with any $\tau'$-measurable operator $x'$, relative to $(\cl A',\tau')$.

We finish this section with two important examples of the previous notions.

\begin{example} $\cl A=L^\infty(X,m)$, for a $\sigma$-finite measure space $(X,m)$. In this case, $\tau$ is integration by the measure $m$
and, for any $p\geq 1$, $L^p(\cl A,\tau)=L^p(X,m)$.
In particular, the standard representation is given by the pointwise action of $L^\infty(X,m)$ on $L^2(X,m)$.
For a non-negative element $f\in L^1(X,m)$, its singular value function $\mu_t(f)$ satisfies
$$\mu_t(f)=\inf\left\{\strut s\geq 0\mid m(\{x\in X\mid f(x)>s\})\leq t\right\};$$
in other words, $t\mapsto \mu_t(f)$ is the non-increasing rearrangement of $f$.
\end{example}

\begin{example} $\cl A=\cl B(H)$ for a Hilbert space $H$. Here, the (essentially unique) normal semi-finite faithful trace $\tau$ is the canonical trace $\tr$.
In this case, for $p\geq 1$, the space $L^p(\cl A,\tau)$ coincides with the Schatten $p$-class $\cl S_p(H)$.
In particular, the standard representation of $\cl B(H)$ is given by the left multiplication action on the Hilbert--Schmidt operators $\cl S_2(H)$. Equivalently, fixing a unitary equivalence $\cl S_2(H)\cong H\ten \overline{H}$
(where $\overline{H}$ is the conjugate Hilbert space of $H$),
the standard representation is the canonical action of $\cl B(H)\ten 1_{\overline{H}}$ on $H\ten\overline{H}$.
Given a positive element $\rho\in \cl T(H)=\cl S_1(H)$, its singular value function $\mu_t(\rho)$ satisfies
$$\mu_t(\rho)=\sum_{n=1}^\infty\lambda_n \chi_{[n-1,n)}(t),$$
where $\lambda_n$ is the $n^{th}$ largest eigenvalue of $\rho$ (including multiplicity).
\end{example}

\section{Local Operations and Classical Communication}\label{s_locc}

In this section, we introduce the class of maps that realise local operations and classical communication (LOCC)
in our general setting and establish some of their properties needed in the sequel.
The origins of our approach lie within the development of algebraic quantum field theory
where the framework is typically encoded in two commuting *-subalgebras of a larger $C^*$-algebra.
In this paper,
a bipartite quantum system is given by a Hilbert space $H$,
together with a von Neumann algebra $\cl A\subseteq\cl B(H)$, and its commutant $\cl B:=\cl A'$.
In the language of \cite[\S4]{kmsw}, this forms a simple, bipartite system satisfying Haag duality.

The standard intuition comes from viewing $\cl A$ and $\cl B$ as the observable algebras
of parties Alice and Bob, respectively, which have joint access to a quantum system modelled on the Hilbert space $H$. For example, when $H=H_A\ten H_B$ for Hilbert spaces $H_A$ and $H_B$, then $\cl A=\cl B(H_A)\ten 1_{H_B}$ and $\cl B = 1_{H_A}\ten \cl B(H_B)$ define the canonical bipartite system structure in the tensor product framework.

We now examine a suitable generalisation of local operations and classical communication (LOCC) in this general bipartite setting,
inspired by the approach in \cite{clmow} and 
related to, but slightly different from, the proposed notion in \cite[\S5]{vw}.

\begin{definition}\label{d_locc} 
Let $\cl A\subseteq\cl B(H)$ be a von Neumann algebra on a Hilbert space $H$, and let $\B=\A'$.
\begin{enumerate}[(i)]
\item A  \emph{one-way right local map relative to $\cl A$} is  a normal completely positive map 
$\Theta : \cl B(H)\to \cl B(H)$ of the form 
$\Theta =  \Phi  \circ\Psi$, where
$\Phi\in \ucp^\sigma_{\cl A}(\cl B(H))$ and $\Psi\in \cp^\sigma_{\cl B}(\cl B(H))$.
Similarly, a  \emph{one-way left local map relative to $\cl A$} is  a normal completely positive map 
$\Theta : \cl B(H)\to \cl B(H)$ of the form 
$\Theta =  \Phi  \circ\Psi$, where
$\Phi\in \cp^\sigma_{\cl A}(\cl B(H))$ and $\Psi\in \ucp^\sigma_{\cl B}(\cl B(H))$.

\item An \emph{instrument} is a collection $\cl I = (\Theta_k)_{k\in \bb{N}}$ of normal completely 
positive maps on $\cl B(H)$ such that, for every $x\in \cl B(H)$, the series
$\sum_{k=1}^{\infty}  \Theta_k(x)$ converges in the weak* topology to  a limit, say $\Theta_{\cl I}(x)$, and  the map 
$x\to \Theta_{\cl I}(x)$ is a normal unital completely  positive map. 
In this case, we sometimes write $\sum_{k=1}^{\infty}\Theta_k$ to denote the map $\Theta_{\cl I}$.
We will identify two instruments if they differ only by a bijective relabelling of the index set.

\item A  \emph{one-way right instrument relative to $\cl A$} is an instrument $\cl I = (\Theta_k)_{k\in \bb{N}}$,
where each of the maps $\Theta_k$ is a one-way right local map relative to $\cl A$.
A  \emph{one-way right LOCC map relative to $\cl A$} is  a normal completely positive map 
$\Theta : \cl B(H)\to \cl B(H)$ of the form $\Theta = \Theta_{\cl I}$, where  
$\cl I$ is a one-way right instrument relative to $\cl A$. 

Similarly, a \emph{one-way left instrument relative to $\cl A$} is an instrument $\cl I = (\Theta_k)_{k\in \bb{N}}$,
where each of the maps $\Theta_k$ is a one-way left local map relative to $\cl A$, and 
a  \emph{one-way left LOCC map relative to $\cl A$} is a normal completely positive map 
$\Theta : \cl B(H)\to \cl B(H)$ of the form $\Theta = \Theta_{\cl I}$, where  
$\cl I$ is a one-way left instrument relative to $\cl A$. 

\item An instrument $(\Gamma_j)_{j\in \bb N}$ is a \emph{coarse-graining} of an instrument $(\Theta_k)_{k\in \bb N}$ if there is a partition $\bN=\bigcup_{j\in \bN}S_j$ so that $\Gamma_j=\sum_{k\in S_j}{\Theta_k}$ for $j\in \bb N$, where each series converges point-weak*.

\item An instrument $\cl I$ is called \emph{one-way local relative to $\cl A$} if $\cl I$ is 
either a one-way right instrument relative to $\cl A$ or a one-way left instrument relative to $\cl A$. We say that an instrument $\cl J$ is \emph{linked} to an instrument $\cl I =  (\Theta_k)_{k\in \bb{N}}$ if 
there exist one-way instruments $(\Theta_{ki})_{i\in \bb{N}}$, $k\in \bb{N}$, such  that  
$\cl J$ is a coarse-graining of the instrument $(\Theta_k\circ \Theta_{ki})_{i,k\in \bb{N}}$. 

\item A map $\Theta\in \ucp^\sigma(\cl B(H))$ is an \emph{LOCC map relative to $\cl A$}  if 
there exists a sequence $(\cl I_0, \dots,\cl I_n)$ of instruments  such  that $\cl I_0$ is a 
one-way local instrument relative to $\cl A$, $\cl I_{l+1}$ is linked to $\cl I_l$, $l = 0,\dots,n-1$, 
and $\Theta = \Theta_{\cl I_n}$.

\end{enumerate}
We denote by $\locc(\cl A)$ the set of all LOCC maps relative to $\cl A$. We also write $\locc^r(\A)$ and $\locc^l(\A)$ for the subsets of $\locc(\A)$ consisting of the one-way right and left LOCC maps relative to~$\A$, respectively. 
Thus, any $\Theta$ in $\locc^r(\A)$ is given by a point weak-* convergent series
\begin{equation}\label{eq:locc-r} \Theta=\sum_{k\in \bb N} \Phi_k\circ \Psi_k
\end{equation}
where $\Phi_k\in \ucp^\sigma_{\cl A}(\cl B(H))$ and $\Psi_k\in \cp^\sigma_{\cl B}(\cl B(H))$.
\end{definition}

The locality of an $\locc(\cl A)$ operation is reflected through the bimodule structure of its implementing maps. For example, if $\Theta$ is a map of the form in \eqref{eq:locc-r}, then Alice's local operations are modelled by the maps
$\Psi_k\in\cp^\sigma_{\cl B}(\cl B(H))$, and Bob's local operations by the channels $\Phi_k\in\ucp^\sigma_{\cl A}(\cl B(H))$.  Since the $\Phi_k$ are $\cl A$-bimodule maps, they do not affect any of Alice's observables and, as shown in Proposition~\ref{p_comp}, they admit Kraus decompositions with operators belonging to Bob's observable algebra $\cl B$. A similar intuition is applied for Alice's local operations $\Psi_k$, which admit Kraus decompositions with operators belonging to Alice's observable algebra $\cl A$. In the Schr\"{o}dinger picture, the map $\Theta_*$ is interpreted as the following operation:
\begin{itemize}
\item Alice performs a local measurement system from $\cl A$ with $(\Psi_k)_*(\rho)$ describing the (unnormalised) post-measurement state;
\item Alice sends the result of the measurement, labelled by $k\in\bb{N}$, to Bob;
\item Bob performs the local operation $(\Phi_k)_*$, conditional upon the result $k$ of Alice's measurement.
\end{itemize}
The map $\Theta$ simply models the above protocol in the Heisenberg picture. See Example~\ref{ex_pol} below for a more explicit example.

A general LOCC protocol relative to $\cl A$ is interpreted as a finite-round protocol where each round is as above with the roles of Alice and Bob potentially reversed, and the measurement systems (and therefore subsequent operations) conditional upon the measurement outcomes in the previous rounds. In the case where $H_A$ and $H_B$ are finite dimensional Hilbert spaces, $H = H_A\otimes H_B$, $\cl A = \cl B(H_A)\otimes 1$ and
$\cl B = 1\otimes \cl B(H_B)$, Definition~\ref{d_locc} reduces to the usual notions as described, for example,
in \cite{clmow}, and we recover the usual operational interpretation of bipartite LOCC operations.

\begin{example}\label{ex_pol}
Let $H$ be a Hilbert space,
$\cl A\subseteq\cl B(H)$ be a von Neumann algebra,
and $\{a_k\mid k\in\N\} \subseteq \cl A$
be a countably infinite measurement system, that is, a sequence in $\cl A$
for which $\sum_{k=1}^\infty a_k^*a_k = 1$ in the weak* topology.
For each $k\in\N$, let $\Psi_k=\Ad(a_k^*)$ on $\cl B(H)$ and let $\Phi_k\in \ucp^\sigma_{\cl A}(\cl B(H))$ be a channel.
Then the series
$\Theta = \sum_{k=1}^\infty\Phi_k\circ\Psi_k$ defines a one-way right LOCC map relative to $\cl A$.
Indeed,
$$\sum_{k=1}^N \Phi_k\circ\Psi_k(1) = \sum_{k=1}^N \Psi_k\circ\Phi_k(1) = \sum_{k=1}^N a_k^* a_k \to_{N\to\infty} 1$$
in the weak* topology.
If $x\in \cl B(H)^+$ then $x\leq \|x\|1$ and hence the partial sums
$\sum_{k=1}^N \Phi_k\circ\Psi_k(x)$ are dominated by $\|x\|1$. Since they form an increasing sequence,
the series $\sum_{k=1}^{\infty} \Phi_k\circ\Psi_k(x)$ converges in the weak* topology.
If $x\in \cl B(H)$ is arbitrary then, using polarisation, we can write $x = \sum_{l=1}^4 \lambda_l x_l$, where
$x_l\in \cl B(H)^+$ and $\lambda_l\in \bb{C}$, $l = 1,2,3,4$.
It follows that the partial sums
$$\sum_{k=1}^N \Phi_k\circ\Psi_k(x) = \sum_{k=1}^N \sum_{l=1}^4 \lambda_l \Phi_k\circ\Psi_k(x_l) =
\sum_{l=1}^4 \sum_{k=1}^N  \lambda_l \Phi_k\circ\Psi_k(x_l)$$
form a weak* convergent sequence.

The predual $\Theta_* : \cl T(H)\to \cl T(H)$ of the map $\Theta$ is given by
$$\Theta_*(\rho) = \sum_{k=1}^\infty(\Phi_k)_*(a_k\rho a_k^*), \ \ \ \rho\in \cl T(H),$$
so one can think of $\Theta_*$ as a protocol where
Alice makes a measurement corresponding to the system $\{a_k\mid k\in\N\}$,
sends the result $k$ to Bob, who then applies $(\Phi_k)_*$. Below we show that any one-way right LOCC map relative to $\cl A$ is of this form, similar to the finite-dimensional setting.

The use of measurements with countably many outcomes is natural in an infinite-dimensional context.
Indeed, even the measurement of an observable modelled by a (possibly unbounded) self-adjoint operator
with continuous spectrum can, within an arbitrarily small amount of error,
be modelled by an observable with countably many disjoint outcomes \cite[\S III.3]{vN}.
\end{example}

\begin{remark}\label{r_locc} Let $\Theta\in \LOCC^r(\A)$ and let $\Phi_k$, $\Psi_k$ be as in~\eqref{eq:locc-r}. Then
$\Phi_k(a) = a \Phi_k(1) = a$, $a\in \cl A$, $k\in\N$, and, in the weak* topology we have
\begin{equation}\label{eq_psi1}
1 = \Theta(1) = \sum_{k=1}^\infty \Psi_k\circ \Phi_k(1) = \sum_{k=1}^{\infty} \Psi_k(1).
\end{equation}
It follows that, for every positive element $x\in \cl B(H)$, the partial sums
$\sum_{i=1}^m \Psi_k(x)$ are norm bounded; since they form an increasing sequence, they converge in the weak* topology.
Using polarisation, we conclude, as in Example~\ref{ex_pol}, that
the sequence $\left(\sum_{i=1}^m \Psi_k(x)\right)_{m\in \bb{N}}$
converges in the weak* topology for every $x\in \cl B(H)$.

\end{remark}

A notion of LOCC operation for general bipartite systems was proposed by Verch--Werner in \cite[\S5]{vw}.
There, a bipartite system is modelled by commuting unital $C^*$-subalgebras $\cl A$ and $\cl B$ of an ambient unital $C^*$-algebra $\cl C$. In this setting, they defined a one-way right LOCC map between bipartite systems
$(\cl A_1,\cl B_1, \cl C_1)$ and $(\cl A_2,\cl B_2, \cl C_2)$, where
$\cl A_i$ and $\cl B_i$ are commuting C*-subalgebras of $\cl C_i$, $i = 1,2$,
by a UCP map $\Theta:\cl C_1\rightarrow\cl C_2$, for which there exist finitely many
completely positive maps $\Psi_k:\cl A_1\rightarrow\cl A_2$ and UCP maps $\Phi_k:\cl B_1\rightarrow\cl B_2$ satisfying
\begin{equation}\label{e:vw}\Theta(ab)=\sum_{k}\Psi_k(a)\Phi_k(b), \ \ \ a\in\cl A_1, \ b\in\cl B_1.\end{equation}
In the special case when $\cl C_1=\cl C_2=\cl B(H)$, $\cl A_1=\cl A_2 =: \cl A$ is a von Neumann algebra, and $\cl B_1=\cl B_2=\cl A'$,
our definition of a one-way right LOCC map relative to $\cl A$ satisfies this condition, albeit, allowing
a countable summation over a \lq\lq classical'' index $k$.
This follows from the fact that any normal
completely positive $\cl A'$-bimodule map on $\cl B(H)$ admits a Kraus decomposition with operators from $\cl A$
(see e.g.~\cite{blecher_smith,haagerup}), and similarly for $\cl B$.

If, in addition, one assumes that $\cl A$ and $\cl B$ are injective factors, then by \cite[Theorem 4.2]{cs}, any
completely positive map $\Psi:\cl A\rightarrow \cl A$ admits a net of
completely positive elementary operators $\Psi_i:\cl A\rightarrow\cl A$
(i.e., operators admitting finitely many Kraus operators from $\cl A$)
satisfying $\norm{\Psi_i}_{cb} \leq \norm{\Psi}_{cb}$ and $\Psi_i\rightarrow\Psi$ in the point weak* topology of $\cbm(\cl A)$.
The maps $\Psi_i$ admit canonical extensions to maps in $\cp^\sigma_{\cl A'}(\cl B(H))$ (through their finite Kraus decompositions), so that we may approximate the
(potentially non-normal) completely positive maps $\Psi_k$ occurring in \eqref{e:vw} by normal maps satisfying our bimodule requirements.
Similar considerations hold for the maps $\Phi_k$.
Hence, in the case of injective factors, one may view the proposed definition of
Verch--Werner as a \lq\lq limit case'' of ours. Note that by \cite[Remark 4.3]{cs},
when $\cl A$ is an injective factor of type II or III, it is \textit{not} true that every \textit{normal} completely positive map
$\Psi : \cl A\rightarrow\cl A$ extends to a \textit{normal} completely positive map $\widetilde{\Psi}\in\cp^\sigma_{\cl A'}(\cl B(H))$.

\begin{proposition}\label{p_comp}
  Let $\A\subseteq \BH$ be a von Neumann algebra on a separable Hilbert space~$H$.
  \begin{enumerate}[(i)]
  \item The class $\LOCC^r(\A)$ of one-way right LOCC maps is closed under
    finite compositions.
  \item  The maps $\Psi_k$ in~\eqref{eq:locc-r} can be taken of the form
$\Psi_k = \Ad(a_k^*)$, for some $a_k\in \cl A$, $k\in \bb{N}$,
with $\sum_{k=1}^{\infty} a_k^*a_k = 1$ in the weak* topology.
\item  The maps $\Phi_k$ in~\eqref{eq:locc-r} can be taken of the form $\Phi_k=\sum_{i=1}^\infty \Ad(c_{ki}^*)$, a point weak*-convergent series, for some $c_{ki}\in \A'$, $k,i\in \bN$, with $\sum_{i=1}^\infty c_{ki}^*c_{ki}=1$ in the weak* topology for every $k\in \bN$.
\end{enumerate}
\end{proposition}

\begin{proof}
Let, as before, $\cl B = \cl A'$.
We first claim that if $\Phi \in \cp_{\cl A}^\sigma(\cl B(H))$ and $\Psi \in \cp_{\cl B}^\sigma(\cl B(H))$,
then
\begin{equation}\label{eq_comm}
\Phi \circ\Psi = \Psi \circ\Phi.
\end{equation}
Indeed, since~$H$ is separable, by
\cite{haagerup} (see also \cite{blecher_smith}), there exists
a bounded column operator $(a_{i})_{i\in \bb{N}}$ with entries in $\cl A$, such that
$$\Psi(x) = \sum_{i=1}^{\infty} a_{i}^* x a_{i}, \ \ \ x\in \cl B(H),$$
where the series converges in the weak* topology.
For every $x\in \cl B(H)$ we now have
$$\Phi\circ\Psi(x) = \Phi \left(\sum_{i=1}^{\infty} a_{i}^* x a_{i}\right)
= \sum_{i=1}^{\infty} \Phi(a_{i}^* x a_{i}) = \sum_{i=1}^{\infty} a_{i}^* \Phi(x) a_{i}=\Psi\circ \Phi(x),$$
showing \eqref{eq_comm}.

Let $\Theta \in \LOCC^r(\A)$, with corresponding maps
$\Phi_k\in \UCP^\sigma_{\A}(\BH)$ and $\Psi_k\in \CP^\sigma_{\B}(\BH)$ for $k\in \bN$, as in~\eqref{eq:locc-r}.
Write
$$\Psi_k(x) = \sum_{i=1}^{\infty} a_{ki}^* x a_{ki}, \ \ \ x\in \cl B(H),$$
for some bounded column operator $(a_{ki})_{i\in \bb{N}}$ with entries in $\cl A$~\cite{haagerup}.
By~\eqref{eq_psi1}, $\Psi_k(1) \leq 1$, and hence the column operator $(a_{ki})_{i\in \bb{N}}$
is contractive. Thus, $a_{ki}$ is a contraction for all $k,i\in \bb{N}$.
Set $\Phi_{ki} = \Phi_k$ for all $k,i\in \bb{N}$ and note that, in the weak* topology, we have
$$\Theta(x) = \lim_{p\to\infty} \lim_{q\to\infty} \sum_{k=1}^p \sum_{i=1}^q \Phi_{ki}\circ \Ad(a_{ki}^*)(x), \ \ \ x\in \cl B(H).$$
Let $x\in \cl B(H)^+$ and $\rho\in \cl T(H)^+$. Then the double limit
\begin{equation}\label{eq_doub}
\lim_{p\to\infty} \lim_{q\to\infty} \sum_{k=1}^p \sum_{i=1}^q
\revduality{\Phi_{ki}\circ \Ad(a_{ki}^*)(x)}{\rho}
\end{equation}
exists. Since the terms of the sequence in \eqref{eq_doub} are positive,
the limit
$$\lim_{L\to\infty}  \sum_{k,i=1}^L \revduality{\Phi_{ki}\circ \Ad(a_{ki}^*)(x)}{\rho}$$
exists. Thus, the partial sums
$\sum_{k,i=1}^L  \Phi_{ki}\circ \Ad(a_{ki}^*)(x)$ converge in the weak* topology for every $x\in \cl B(H)^+$
and hence, by the polarisation identity, for every $x\in \cl B(H)$.
It is now straightforward to see that the limit coincides with $\Theta(x)$.
Note, moreover, that identity \eqref{eq_psi1} shows that
$\sum_{k,i=1}^{\infty} a_{ki}^* a_{ki} = 1$ in the weak* topology.
This establishes~(ii). The assertion in (iii) follows by considering the Kraus decomposition
of the maps $\Phi_k$.

To show~(i), assume that $\Theta_i\in \LOCC^r(\A)$ and
let $\Phi_k^{(i)}$, $\Psi_k^{(i)}$, $k\in \bb{N}$, be the maps as in~\eqref{eq:locc-r},
associated with $\Theta_i$, $i = 1,2$.
By \eqref{eq_comm}, for every $x\in \cl B(H)$, we have that, in the weak* topology,
\begin{eqnarray*}
(\Theta_1\circ \Theta_2)(x)
& = &
\lim_{p\to \infty} \lim_{q\to \infty} \sum_{k=1}^p \sum_{l=1}^q \left(\Phi_k^{(1)}\circ \Psi_k^{(1)} \circ \Phi_l^{(2)}
\circ \Psi_l^{(2)}\right)(x)\\
& = &
\lim_{p\to \infty} \lim_{q\to \infty} \sum_{k=1}^p \sum_{l=1}^q
\left(\Phi_k^{(1)}\circ \Phi_l^{(2)}  \circ  \Psi_k^{(1)}\circ \Psi_l^{(2)}\right)(x).
\end{eqnarray*}
Set $\Phi_{kj} = \Phi_k^{(1)}\circ \Phi_l^{(2)}$ and $\Psi_{kj} = \Psi_k^{(1)}\circ \Psi_l^{(2)}$.
An argument similar to the one in the previous paragraph now implies that
$$(\Theta_1\circ \Theta_2)(x) = \sum_{k,l=1}^{\infty}  (\Phi_{kj} \circ \Psi_{kj})(x), \ \ \ x\in \cl B(H),$$
in the weak* topology, so $\Theta_1\circ \Theta_2\in \LOCC^r(\A)$.
\end{proof}

\begin{remark}\label{r_mirror} (i) The expression of one-way right LOCC maps given in Proposition~\ref{p_comp}\,(ii)
reflects the notion of fine-graining of LOCC channels described in~\cite{clmow}.

(ii)  By symmetry, observations analogous to those above for one-way right LOCC maps
also hold for the one-way left LOCC maps.
\end{remark}

\section{State Convertibility via $\locc(\A)$}\label{s:conv}

Having established an appropriate generalisation of LOCC operations in the preceding section,
we now define the corresponding notions of convertibility.

\begin{definition}\label{d_conv}
Let $\cl A\subseteq\cl B(H)$ be a von Neumann algebra on a Hilbert space $H$, and let $\psi, \vphi\in H_1$.
\begin{enumerate}[(i)]
\item We say that $\psi$ \emph{is convertible to} $\vphi$ \emph{via} $\locc(\cl A)$ if
there exists $\Theta\in \locc(\cl A)$ such that $\Theta_*(\om_\psi) = \om_\vphi$.

\item We say that $\psi$ \emph{is approximately convertible to} $\vphi$ \emph{via} $\locc(\cl A)$
if for every $\varepsilon > 0$ there exists $\Theta\in \locc(\cl A)$ such that $\nor{\Theta_*(\om_\psi)- \om_\vphi}<\ep$.
\end{enumerate}
We also make analogous definitions with $\locc^l(\A)$ and $\locc^r(\A)$ in place of $\locc(\A)$.
\end{definition}

The goal of the rest of this section is to show that approximate convertibility can be realised by using only one-way LOCC maps
(see Corollary~\ref{c_oneway}). This generalises to the commuting operator framework
a result of Lo-Popescu \cite{lp} for finite-dimensional bipartite systems.
We note that another generalisation of this theorem was established in \cite{obnm}, which we recover from our results
by taking the special case $\A=\BH$ in its standard representation, for $H$ a separable Hilbert space.
 The essential features of our argument, similar to the finite-dimensional case, are the symmetry and polar decomposition induced from the standard
form of a von Neumann algebra \cite{haagerup2}, as highlighted in Section~\ref{s_bc}. We begin with a few preparatory results.

Let $H$ be a Hilbert space, $\cl A\subseteq\cl B(H)$ be a (semi-finite) von Neumann algebra equipped with a
normal semi-finite faithful trace $\tau$, and $\psi\in H$.
To avoid double subscripts, we will write $\rho_\psi\in L^1(\A,\tau)$ for the density operator
$\rho_{\omega_\psi}$ arising from the restriction of the vector state $\om_\psi$ to $\cl A$.
In the case where $H=L^2(\A,\tau)$ and $\cl A$ is in standard form,
we write $\rho_\psi'\in L^1(\A',\tau')$ for the density operator affiliated to $\cl A'$ satisfying
$$\ip{a'\psi}{\psi}=\om_{\psi}|_{\cl A'}(a')=\tau'(\rho'_\psi a'), \ \ \ a'\in \cl A',$$
where $\tau'=\tau\circ R$ is the canonical trace on $\cl A'$.
Recall that for $t > 0$, we write $\mu'_t(\rho'_\psi)=\mu_t(\rho'_\psi;\A',\tau')$ for the $t$-th singular value of $\rho'_\psi$
relative to $(\A',\tau')$.

\begin{lemma}\label{l_Schmidt}
Let $(\cl A,\tau)$ be a semi-finite von Neumann algebra, represented in its standard form on $L^2(\A,\tau)$,
and let $\psi\in L^2(\cl A,\tau)$. Then \[J\rho_\psi'J=\rho_{\psi^*}\qand \mu'_t(\rho_\psi')=\mu_t(\rho_{\psi^*})=\mu_t(\rho_\psi),\quad t>0.\]
\end{lemma}

\begin{proof}
  First we show that $\rho_{\psi^*}=J\rho'_\psi J$ and $\mu_t'(\rho_\psi')=\mu_t(\rho_{\psi^*})$. Recall that $\psi^*=J\psi$ and that $J$ is a conjugate linear isometry satisfying~(\ref{eq_jaj}) with $J^2=1$.
For $a\in \A$, we therefore have
\begin{align*} \tau(\rho_{\psi^*}a)&=\ip{aJ\psi}{J\psi}=\ip{\psi}{JaJ\psi}=\ip{Ja^*J\psi}{\psi}=\tau'(\rho'_\psi Ja^*J)\\&=\tau\left(J(JaJ\rho_\psi')J\right)=\tau(J\rho_\psi'Ja).
\end{align*}
Hence,
\begin{equation}\label{eq_psiJ}
\rho_{\psi^*} = J\rho'_\psi J.
\end{equation}
Note that $R\colon \A'\to \A$ induces a trace-preserving bijection $\P(\A')\to \P(\A)$, $p'\mapsto Jp'J$. For $t > 0$, using \eqref{eq_psiJ} we obtain
\begin{align*}
  \mu'_t(\rho_\psi')&=\inf\{\nor{\rho_\psi' p'}\colon {p'\in \P(\A'),\,\tau'(1-p')\le t}\}\\
  &=\inf\{\nor{\rho_\psi'JpJ}\colon {p\in \P(\A),\,\tau(1-p)\le t}\}\\
  &=\inf\{\nor{J\rho_\psi'Jp}\colon {p\in \P(\A),\,\tau(1-p)\le t}\}\\
  &=\inf\{\nor{\rho_{\psi^*}p}\colon {p\in \P(\A),\,\tau(1-p)\le t}\} = \mu_t(\rho_{\psi^*}).
\end{align*}

Now, by \cite[Exercises IX.1.2--3]{t2}, there exists a partial isometry $u\in\cl A$
such that $|\psi^*|=\psi u^*$, and $\psi^*=u^*|\psi^*|$.
It follows that
$$\psi^* = u^*\psi u^* = u^* |\psi^*| = u^* J |\psi^*| = u^* J u \psi^* =  u^*JuJ\psi.$$
Since $u\in\cl A$, we have $\rho_{\psi*}=\rho_{u^*JuJ\psi}=u^*\cdot\rho_{JuJ\psi}\cdot u$.
By \cite[Lemma 2.5\,(vi)]{fk},
\begin{equation}\label{eq_mut}
\mu_t(\rho_{\psi*})\leq\nor{u^*}\nor{u}\mu_t(\rho_{JuJ\psi})\leq\mu_t(\rho_{JuJ\psi}),
\end{equation}
for $t > 0$. But $JuJ$ is a contraction in $\cl A'$, so $\om_{JuJ\psi}|_{\cl A}\leq \om_{\psi}|_{\cl A}$, that is, $\rho_{JuJ\psi}\leq \rho_{\psi}$. By \cite[Lemma 2.5\,(iii)]{fk}, $\mu_t(\rho_{JuJ\psi})\leq\mu_t(\rho_\psi)$, $t>0$.
Thus, by \eqref{eq_mut}, $\mu_t(\rho_{\psi^*})\leq\mu_t(\rho_{\psi})$, for $t>0$. By symmetry, we obtain equality.
\end{proof}

\begin{proposition}\label{p_lopopescu}
Let $(\cl A,\tau)$ be a semi-finite factor in its standard form on $L^2(\A,\tau)$, let $\cl B = \cl A'$ and let $\epsilon > 0$.
For any $(\psi,b)\in H\times \B$, there exists a unitary
$u\in \cl B$ and partial isometries $v,w\in \A$ so that $\psi=v^*v\psi$ and, if $z=JbJv$, then
\[ \nor{b\psi}=\nor{z\psi}\qand \nor{b\psi-uwz\psi}<\epsilon.\]
Moreover, the partial isometry $v$ can be chosen independently of $b$.
\end{proposition}

\begin{proof} Let $\psi^*=v|\psi^*|$ be the polar decomposition of $\psi^*$
\cite[Exercises IX.1.2--3]{t2}; thus,
$v\in \cl A$ is a partial isometry with $\psi=v^*|\psi|$, and
the projections $v^*v$ and $vv^*$ have ranges
$\overline{\cl B\psi}$ and $\overline{\cl B|\psi|}$, respectively, so in particular, $\psi=v^*v\psi$.
It follows that
\begin{equation}\label{eq_vpsi}
v\psi=vv^*|\psi|=|\psi|=J|\psi|=Jv\psi.
\end{equation}
Note that $z \in \A$. We have $v^*vb\psi = b v^*v\psi = b\psi$, and hence, using \eqref{eq_vpsi},
\begin{equation*}
\nor{b\psi}=\nor{vb\psi}=\nor{bv\psi}=\nor{bJv\psi}=\nor{JbJv\psi}=\nor{z\psi},
\end{equation*}
as desired.

Note that 
\begin{equation}\label{eq_taudash}
\tau'(\rho') = \tau(J\rho'\mbox{}^{*} J), \ \ \ \rho'\in L^1(\cl B, \tau');
\end{equation}
indeed, the formula holds by the definition of $\tau'$ in the case $\rho' \in \cl B$, and the general case follow by
approximating $\rho'$ by a sequence in $\cl B$ in the norm $\|\cdot\|_1$.
Let $\alpha=z\psi$ and $\beta=b\psi$.  For $c\in\cl B$, by \eqref{eq_taudash}, \eqref{eq_vpsi} and \eqref{eq_jaj},
we have
\begin{align*}
\tau(\rho_\alpha R(c)) &= (R(c)z\psi,z\psi)
  =\ip{Jc^*JJbJv\psi}{JbJv\psi}=\ip{Jc^*bv\psi}{Jbv\psi}\\
  &=\ip{bv\psi}{c^*bv\psi} = \ip{cv^*vb\psi}{b\psi} =\ip{c \beta}{\beta}
    =\tau'(\rho'_\beta c)\\
    &=\tau(J\rho_\beta'J R(c)).
\end{align*}
Thus, by Lemma~\ref{l_Schmidt}, $\rho_\alpha=J\rho_\beta'J=\rho_{\beta^*}$
and
$$\mu'_t(\rho'_{\alpha})=\mu_t(\rho_{\alpha})=\mu_t(\rho_{\beta^*})=\mu_t'(\rho'_{\beta}), \ \ \ t>0.$$
Since $\A$ is a factor,
by \cite[Theorem 3.4\,(1)]{h} for every $\ep>0$ there exists a unitary $u\in\cl B$ such that
$$\nor{\rho'_{\beta}-\rho'_{u\alpha}}_1 = \nor{\rho'_{\beta}-u \rho'_{\alpha} u^*}_1 < \ep^2.$$
By the continuity of Stinespring's representation \cite[Theorem 1]{ksw}, there exist a Hilbert space $K$,
a *-homomorphism $\pi:\cl B\rightarrow \cl B(K)$ and vectors
$\xi,\eta\in K$ such that $\om_{\beta}|_{\cl B}=\om_\xi\circ\pi$ and $\om_{u\alpha}|_{\cl B}=\om_\eta\circ\pi$, and
$$\nor{\xi-\eta}\leq\nor{\omega_\beta|_{\B}-\omega_{u\alpha}|_{\B}}^{1/2}=\nor{\rho'_{\beta}-\rho'_{u\alpha}}_1^{1/2}<\ep.$$
By the uniqueness of Stinespring representations,
there exist partial isometries $w_1 : L^2(\cl A,\tau)\rightarrow K$ and
$w_2 : K\rightarrow L^2(\cl A,\tau)$ such that $w_1c=\pi(c)w_1$ and $w_2\pi(c)=cw_2$ for all $c\in \cl B$,
$w_1u\alpha=\eta$ and $w_2\xi=\beta$. Then $w:=w_2w_1$ is a contraction in $\cl A$, and
\begin{equation*}
\nor{b\psi-uwz\psi}=\nor{\beta-wu\alpha} = \nor{w_2\xi-w_2\eta} < \ep.\qedhere
\end{equation*}
\end{proof}

The following estimate is straightforward and we will make use of it multiple times.

\begin{lemma}\label{l_piso}
Let $\psi\in H$, $\nor{\psi}\leq 1$, and let $v\in\BH$ be a contraction. If $\nor{\psi}-\nor{v\psi}<\ep$, then $\nor{(1-v^*v)^{1/2}\psi}<\sqrt{2\ep}$.
\end{lemma}

\begin{proof} The assumption implies
\begin{align*}\nor{(1-v^*v)^{1/2}\psi}^2&=\ip{\strut(1-v^*v)\psi}{\psi}=\nor{\psi}^2-\nor{v\psi}^2\\
&=(\nor{\psi}+\nor{v\psi})(\nor{\psi}-\nor{v\psi})
\leq 2(\nor{\psi}-\nor{v\psi})
<2\ep.\qedhere\end{align*}
\end{proof}

\begin{lemma}\label{l_dini}
Let $\cl A$ be a von Neumann algebra equipped with a normal faithful semi-finite trace $\tau$, and let
$(\rho_k)_{k\in \bb{N}}$ be an increasing sequence of self-adjoint elements of $L^1(\cl A,\tau)$.
If $\rho_k\to_{k\to \infty}\rho$ in the weak topology, for some $\rho\in L^1(\cl A,\tau)$,
then $\rho_k\to_{k\to \infty}\rho$ in norm.
\end{lemma}
\begin{proof}
We have that $\rho_k\leq \rho$; thus, $\rho - \rho_k \geq 0$ for every $k$. Thus,
\[\|\rho - \rho_k\| = (\rho - \rho_k)(1) \to_{k\to\infty} 0.\qedhere\]
\end{proof}

\begin{lemma}\label{l_ineq}
Let $H$ be a separable Hilbert space and $\Phi:\BH\rightarrow\BH$ be positive. Then for any self-adjoint $T\in\BH$ and any self-adjoint $\rho\in\cl T(H)$,
$$|\la\Phi(T),\rho\ra|\leq\norm{T}\la\Phi(1),|\rho|\ra.$$
\end{lemma}

\begin{proof} Let $(e_n)_{n=1}^\infty$ be an orthonormal basis of eigenvectors of $\rho$ with corresponding eigenvalues $(\lambda_n)_{n=1}^\infty\subseteq\R$. Since $-\norm{T}1\leq T\leq\norm{T}1$, by positivity of $\om_{e_n}\circ\Phi$ we have
$$-\norm{T}(\Phi(1)e_n,e_n)\leq(\Phi(T)e_n,e_n)\leq\norm{T}(\Phi(1)e_n,e_n),$$
so that $|(\Phi(T)e_n,e_n)|\leq\norm{T}(\Phi(1)e_n,e_n)$ for all $n\in\N$. Hence
\begin{align*}|\la\Phi(T),\rho\ra|&=\bigg|\sum_{n=1}^\infty\lambda_n(\Phi(T)e_n,e_n)\bigg|\leq\sum_{n=1}^\infty|\lambda_n||(\Phi(T)e_n,e_n)|\\
&\leq\sum_{n=1}^\infty|\lambda_n|\norm{T}(\Phi(1)e_n,e_n)=\norm{T}\la\Phi(1),|\rho|\ra.\qedhere\end{align*}
\end{proof}

We now begin the proof that approximate convertibility can be achieved by one-way $\locc(\cl A)$ maps (see Corollary~\ref{c_oneway}). We first establish the case when $\cl A$ is standardly represented using induction on the number of rounds of the $\locc(\cl A)$ protocol. The base case of the induction is isolated into the next result.

\newcommand{\ser}[1]{\sum_{#1=1}^\infty}
\begin{proposition}\label{prop_p1} 
Let $(\cl A,\tau)$ be a semi-finite factor in its standard form on $H=L^2(\A,\tau)$.
Suppose that $\cl I=(\Theta_k)_k$ is a one-way right instrument relative to~$\cl A$ and $\cl J=(\Lambda_k)_k$ is an instrument which is linked to $\cl I$. 
For any $\psi\in H$ 
and any $\ep > 0$, there exists an instrument $(\Gamma_{k})_{k}$ which is a coarse-graining of some one-way right instrument relative to~$\cl A$, such that
\begin{equation}
  \sum_{k=1}^\infty\norm{\Lambda_{k*}(\om_\psi)-\Gamma_{k*}(\om_\psi)}<\ep.\label{eq:P(1)}
\end{equation}
\end{proposition}
\begin{proof}
  We may assume that $\nor{\psi}\le 1$. Each $\Theta_k$ is one-way right local, so  by Proposition~\ref{p_comp}\,(ii) we may write $\Theta_k=\sum_{j,l=1}^\infty \Ad(a_{kj}^*)\circ\Ad(c_{kl}^*)$, where $a_{kj}\in\cl A$ satisfy $\sum_{k=1}^\infty\sum_{j=1}^{\infty}a_{kj}^*a_{kj}=1$ and $c_{kl}\in\cl B$ satisfy $\sum_{l=1}^\infty c_{kl}^*c_{kl}=1$ for each $k\in\N$. We define $\psi_{kj}:=a_{kj}\psi\in H$, for $k,j\in \bb N$.

Since $\cl J$ is linked to $\cl I$, the instrument~$\cl J$ is a coarse-graining of an instrument of the form $(\Theta_k\circ \Theta_{ki})_{k,i}$, for a collection of one-way instruments $\cl J_k=(\Theta_{ki})_i$ indexed by $k\in\N$. Write $L=\{k\in\N\mid \cl J_k \ \textnormal{is one-way left}\}$ and $R=\{k\in\N\mid \cl J_k \ \textnormal{is one-way right}\}$.

Suppose first that $k\in L$. By Proposition~\ref{p_comp}\,(ii) and Remark~\ref{r_mirror}\,(ii), we may assume that $\Theta_{ki}=\Psi^L_{ki}\circ\Ad(b_{ki}^*)$, where $\Psi^L_{ki}\in\ucp^\sigma_{\cl B}(\cl B(H))$, $b_{ki}\in\cl B$ and $\sum_{i=1}^\infty b_{ki}^*b_{ki}=1$.
Hence, in the point weak*-topology,  we have
\begin{equation}\label{e_rl}
\Theta_k\circ \Theta_{ki}=\ser j \ser l \Ad((b_{ki}c_{kl}a_{kj})^*)\circ \Psi^L_{ki}.
\end{equation}
Since~$\A$ is a factor, we can apply Proposition~\ref{p_lopopescu} to the pairs $(\psi_{kj},b_{ki}c_{kl})\in H\times \B$.
We obtain  unitaries $u_{kijl}\in \B$ and partial isometries $v_{kj},w_{kijl}\in \A$ so that $\psi_{kj}=v_{kj}^*v_{kj}\psi_{kj}$ and $z_{kijl}=Jb_{ki}c_{kl}Jv_{kj}\in \A$ satisfy $\nor{b_{ki}c_{kl}\psi_{kj}}=\nor{z_{kijl}\psi_{kj}}$ and
\begin{equation}\label{e_bound1}
\|b_{ki}c_{kl}\psi_{kj}-u_{kijl}w_{kijl}z_{kijl}\psi_{kj}\|<\frac{\epsilon}{2^{i+j+k+l+2}}
\end{equation}
for $i,j,k,l\in \bN$.
Let $a_{kijl}=w_{kijl}z_{kijl}\in \A$, and
observe that
$$\nor{z_{kijl}\psi_{kj}}-\norm{a_{kijl}\psi_{kj}}=\nor{b_{ki}c_{kl}\psi_{kj}}-\norm{u_{kijl}w_{kijl}z_{kijl}\psi_{kj}}<\frac{\ep}{2^{i + j + k+l+2}}.$$
Let $\widetilde{a_{kijl}} := (1-w_{kijl}^*w_{kijl})z_{kijl}\in \A$. By Lemma~\ref{l_piso}, the preceding inequality implies
\begin{equation}\label{eq_kij1}
\nor{\widetilde{a_{kijl}}\psi_{kj}}=\nor{(1-w_{kijl}^*w_{kijl}) z_{kijl}\psi_{kj}}<\sqrt{\frac{\ep}{2^{i + j+k+l+1}}}.
\end{equation}

Now suppose that $k\in R$. The map $\Theta_{\cl J_k}=\sum_{i=1}^\infty \Theta_{ki}$ is then a one-way right LOCC map relative to~$\cl A$, with $\Theta_{ki}=\Psi^R_{ki}\circ\Phi_{ki}$, where $\Psi^R_{ki}\in\cp^\sigma_{\cl B}(\cl B(H))$ and $\Phi_{ki}\in\ucp^\sigma_{\cl A}(\cl B(H))$. Then
\begin{equation}\label{eq_kR}\Theta_k\circ \Theta_{ki}=\sum_{j=1}^\infty\sum_{l=1}^\infty \Ad(a_{kj}^*)\circ\Psi^R_{ki}\circ(\Ad(c_{kl}^*)\circ\Phi_{ki})\end{equation}
in the point weak*-topology.

Recall that $v_{kj}$ was defined above for $k\in L$ and $j\in \bb N$; we define $v_{kj}:=1$ for $k\in R$ and $j\in\N$. Then
$$A:=\sum_{k=1}^\infty\sum_{j=1}^\infty a_{kj}^*v_{kj}^*v_{kj}a_{kj}\leq\sum_{k=1}^\infty\sum_{j=1}^\infty a_{kj}^*a_{kj}=1;$$
thus, the operator $(1-A)^{1/2}$ is a well-defined element of $\cl A$.
Moreover, since $a_{kj}\psi=\psi_{kj}=v_{kj}^*v_{kj}\psi_{kj}=v_{kj}^*v_{kj}a_{kj}\psi$ for all $j,k\in\N$,
we have
$\psi=\sum_{k=1}^\infty\sum_{j=1}^\infty a_{kj}^*a_{kj}\psi=A\psi$,
showing that $\psi\in\Ker(1-A)$ and thus that $\psi \in \Ker((1-A)^{1/2})$.
In particular,
\begin{equation}\label{eq_1-A}
\left(\Ad((1-A)^{1/2})\right)_*(\omega_\psi) = 0.
\end{equation}

Fix a probability distribution $(p_n)$ over $\N$ with $p_n > 0$ for each $n$. For $k\in L$ and $i,j,l\in \bb N$, define
$\Gamma^L_{kijl}\in \cp_{\cl B}^\sigma(\cl B(H))$ for $i\in\bN$ by
\[\Gamma^L_{kijl}=\left(\Ad(a_{kj}^*)\circ (\Ad (a_{kijl}^*)+\Ad(\widetilde{a_{kijl}}^*)) + p_kp_ip_j p_l\Ad((1-A)^{1/2})\right)\circ \Psi^L_{ki}.\]
Since
$$a_{kijl}^*a_{kijl}+\widetilde{a_{kijl}}^*\widetilde{a_{kijl}}=z_{kijl}^*z_{kijl}=v_{kj}^*Jc_{kl}^*b_{ki}^*b_{ki}c_{kl}Jv_{kj}$$
and $\ser i \ser l c_{kl}^*b_{ki}^*b_{ki}c_{kl}=1$ for every $k\in L$, in the weak* topology we have
\begin{align*}\sum_{k\in L}\sum_{i=1}^\infty\ser j\ser l\Gamma^L_{kijl}(1)
&=\sum_{k\in L}\sum_{i=1}^\infty\sum_{j=1}^\infty\sum_{l=1}^\infty a_{kj}^*z_{kijl}^*z_{kijl}
                                                                                a_{kj}+p_kp_ip_jp_l(1-A)\\
&=\sum_{k\in L}\sum_{j=1}^\infty a_{kj}^*v_{kj}^*v_{kj}a_{kj} +p_kp_j(1-A).
\end{align*}
For $k\in R$, define one-way right local maps $\Gamma^R_{ki}$ by
$$\Gamma^R_{ki}=\sum_{j=1}^\infty\sum_{l=1}^\infty(\Ad(a_{kj}^*)\circ\Psi^R_{ki}+p_kp_ip_j\Ad((1-A)^{1/2}))\circ(\Ad(c_{kl}^*)\circ\Phi_{ki}).$$
Then
\begin{align*}\sum_{k\in R}\sum_{i=1}^\infty\Gamma^R_{ki}(1)&=\sum_{k\in R}\sum_{i=1}^\infty\sum_{j=1}^\infty a_{kj}^*\Psi^R_{ki}(1)a_{kj}+p_kp_ip_j(1-A)\\
&=\sum_{k\in R}\sum_{j=1}^\infty a_{kj}^*a_{kj}+p_kp_j(1-A).\\
\end{align*}
Putting things together, we get
\begin{align*}&\sum_{k\in L}\sum_{i=1}^\infty\ser j\ser l\Gamma^L_{kijl}(1)+ \sum_{k\in R}\sum_{i=1}^\infty\Gamma^R_{ki}(1)\\
&=\sum_{k\in L}\sum_{j=1}^\infty a_{kj}^*v_{kj}^*v_{kj}a_{kj} +p_kp_j(1-A)+\sum_{k\in R}\sum_{j=1}^\infty a_{kj}^*a_{kj}+p_kp_j(1-A)\\
&=\sum_{k=1}^\infty\sum_{j=1}^\infty a_{kj}^*v_{kj}^*v_{kj}a_{kj} +p_kp_j(1-A)\\
&=A+(1-A)\\
&=1.
\end{align*}

It follows that the series
$$\sum_{k\in L}\sum_{i=1}^\infty\sum_{j=1}^\infty \sum_{l=1}^\infty \Ad(u_{kijl}^*)\circ \Gamma^L_{kijl}(x)+\sum_{k\in R}\sum_{i=1}^\infty\Gamma^R_{ki}(x)$$
is convergent in the weak* topology for every positive $x\in\cl B(H)$.
By polarisation,
it is convergent in the weak* topology for every $x\in\cl B(H)$. With
$$\Gamma^L_{ki}:=\sum_{j=1}^\infty \sum_{l=1}^\infty \Ad(u_{kijl}^*)\circ \Gamma^L_{kijl},$$
and $\Gamma_{ki}:=\Gamma^L_{ki}$ for $k\in L$, $i\in\N$, and $\Gamma_{ki}:=\Gamma^R_{ki}$ for $k\in R$, $i\in\N$,
it follows that $(\Gamma_{ki})_{k,i}$ is a coarse-graining of a one-way right instrument relative to $\cl A$.

By \eqref{eq_kR} and \eqref{eq_1-A}, for $k\in R$ we have $\Theta_{ki*}(\Theta_{k*}(\om_\psi))=\Gamma_{ki*}(\om_\psi)$, $i\in\N$. Also,
using \eqref{e_rl}, \eqref{e_bound1}, \eqref{eq_kij1} and \eqref{eq_1-A} with the identity $\Ad(a_{kj})\omega_\psi=\omega_{\psi_{kj}}$
and Lemma~\ref{l_dini} consecutively, we have
\begingroup\allowdisplaybreaks
\begin{align*}
&\ser k\ser i\nor{\Gamma_{ki*}(\om_\psi)-\Theta_{ki*}(\Theta_{k*}(\om_\psi))}\\
&\leq\sum_{k\in L}\sum_{i,j,l=1}^\infty\nor{\Psi^L_{ki*} \circ\left(\Ad(u_{kijl})\circ(\Ad(a_{kijl})+\Ad(\widetilde{a_{kijl}}))
  -\Ad(b_{ki}c_{kl}) \right)\omega_{\psi_{kj}}}\\
&\leq\sum_{k\in L}\sum_{i,j,l=1}^\infty\nor{u_{kijl}(a_{kijl}\om_{\psi_k}a_{kijl}^*+\widetilde{a_{kijl}}\om_{\psi_{kj}}\widetilde{a_{kijl}}^*)u_{kijl}^*-b_{ki}c_{kl}\om_{\psi_{kj}}c_{kl}^*b_{ki}}\\
&\leq\sum_{k\in L}\sum_{i,j,l=1}^\infty\nor{u_{kijl}a_{kijl}\om_{\psi_{kj}}a_{kijl}^*u_{kijl}^*-b_{ki}c_{kl}\om_{\psi_{kj}}c_{kl}^*b_{ki}}+\nor{\widetilde{a_{kijl}}\om_{\psi_{kj}}\widetilde{a_{kijl}}^*}\\
&\leq\sum_{k\in L}\sum_{i,j,l=1}^\infty\nor{u_{kijl}a_{kijl}\psi_{kj}-b_{ki}c_{kl}\psi_{kj}}(\nor{a_{kijl}\psi_{kj}}+\nor{b_{ki}c_{kl}\psi_{kj}})+\nor{\widetilde{a_{kijl}}\psi_{kj}}^2\\
&<2\sum_{k\in L}\sum_{i,j,l=1}^\infty\frac{\ep}{2^{k + i + j+l+1}} \leq \ep.
\end{align*}\endgroup
%
Since $\cl J=(\Lambda_k)_k$ is a coarse-graining of $(\Theta_k\circ \Theta_{ki})_{k,i}$, applying the same coarse-graining to $(\Gamma_{ki})_{k,i}$ produces an instrument~$(\Gamma_k)_k$ satisfying~\eqref{eq:P(1)}.
\end{proof}
\begin{theorem}\label{th_lopopescu_std} 
Let $(\cl A,\tau)$ be a semi-finite factor in its standard form on $H=L^2(\A,\tau)$.
For any $\Theta\in\locc(\cl A)$, $\psi\in H$
and $\ep > 0$, there exists $\Theta_\ep \in \LOCC^r(\A)$ such that
$\nor{\Theta_*(\om_\psi)-\Theta_{\ep*}(\om_\psi)} < \ep$.
\end{theorem}

\begin{proof}
  We may assume that $\nor{\psi}\le1$. Any LOCC map $\Theta\in\locc(\cl A)$ is of the form $\Theta=\Theta_{\cl{I}_n}$ where $(\cl{I}_0,\dots,\cl{I}_n)$ is a sequence of instruments such that $\cl{I}_0$ is one-way local relative to $\cl A$, and $\cl I_{l+1}$ is linked to $\cl{I}_{l}$ for each $l=0,\dots,n-1$. Since the trivial instrument $(\id,0,0,\dots)$ is one-way right local and any one-way left local instrument is linked to it, without loss of generality, we may suppose that $\cl{I}_0$ is a one-way right instrument. Writing $\cl I_n=(\Theta_k^{(n)})_{k}$,  consider the following proposition $P(n)$: for every $\ep>0$ there exists an instrument $\cl I_{\ep}=(\Gamma_{k})_{k}$ which is a coarse-graining of some one-way right instrument relative to~$\cl A$, such that
$$\sum_{k=1}^\infty\norm{\Theta_{k*}^{(n)}(\om_\psi)-\Gamma_{k*}(\om_\psi)}<\ep.$$
If $P(n)$ were true for every $n\in\N$ then the Theorem follows with $\Theta_{\ep}=\Theta_{\cl I_{\ep}}$. We therefore proceed by induction on $n$, noting that $P(0)$ is vacuously true and $P(1)$ is precisely Proposition~\ref{prop_p1}.



Assuming $P(n)$ is true, let $\Theta$ be an LOCC map of the form $\Theta=\Theta_{\cl{I}_{n+1}}$ where $(\cl{I}_0,\dots,\cl{I}_{n+1})$ is a sequence of instruments such that $\cl{I}_0$ is a one-way local right instrument relative to $\cl A$,
and $\cl I_{l+1}$ is linked to $\cl{I}_{l}$ for each $l=0,\dots,n$.
Then $\cl I_{n+1}$ is a coarse-graining of $(\Theta_k\circ \Theta_{ki})_{k,i}$, where $\cl I_{n}=(\Theta_k)_k$ and
$\cl J_k=(\Theta_{ki})_i$ is a one-way instrument for each $k\in\N$. Given $\ep>0$,
by $P(n)$ there exists a coarse-graining 
$(\Gamma_{k})_{k}$ of some one-way right instrument such that
$$\sum_{k=1}^\infty\norm{\Theta_{k*}(\om_\psi)-\Gamma_{k*}(\om_\psi)}<\frac{\ep}{2}.$$

We have $\Gamma_k=\sum_{j\in S_k} \tilde \Gamma_j$ for some one-way right instrument $\cl I_0'=(\tilde\Gamma_j)_j$ and some partition $(S_k)_k$ of~$\bb N$. Let $\tilde \Theta_{ji}=\Theta_{ki}$ for $j\in S_k$, $k\in \bb N$. Then $(\tilde \Theta_{ji})_i$ is a one-way instrument for each~$j\in \bb N$, and $(\Gamma_k\circ \Theta_{ki})_{k,i}$ is a coarse-graining of $\cl I_1'=(\tilde \Gamma_j\circ \tilde \Theta_{ji})_{j,i}$, and $\cl I_1'$ is linked to $\cl I_0'$. Applying Proposition~\ref{prop_p1} to the pair $(\cl I_0', \cl I_1')$, we obtain an instrument $(\hat \Gamma_{ji})_{j,i}$ which is a coarse-graining of some one-way right instrument, and satisfies \[\sum_{j,i=1}^\infty\nor{\tilde \Theta_{ji*}( \tilde \Gamma_{j*}(\omega_\psi))-\hat\Gamma_{ji*}(\omega_\psi)}<\frac\epsilon2.\] Setting $\Gamma_{ki}:=\sum_{j\in S_k}\hat \Gamma_{ji}$, we obtain an instrument $(\Gamma_{ki})_{k,i}$ which is a coarse-graining of a one-way right instrument, and satisfies
\begin{align*}
\ser k\ser i\norm{\Theta_{ki*}(\Gamma_{k*}(\om_\psi))-\Gamma_{ki*}(\om_\psi)}&\le 
\ser k\ser i\sum_{j\in S_k}\norm{\tilde\Theta_{ji*}(\tilde \Gamma_{j*}(\om_\psi))-\hat\Gamma_{ji*}(\om_\psi)}\\&=\ser i\ser j \norm{\tilde\Theta_{ji*}(\tilde \Gamma_{j*}(\om_\psi))-\hat\Gamma_{ji*}(\om_\psi)}<\frac{\ep}{2}.
\end{align*}
For each $k,i\in\N$, $\Theta_{ki*}((\Theta_{k*}-\Gamma_{k*})(\om_\psi))\in\cl T(H)$ is self-adjoint, and attains its norm on self-adjoint operators in $\BH$. Hence, there exists $T_{ki}\in\BH$ with $T_{ki}=T_{ki}^*$ and $\norm{T_{ki}}\leq 1$ satisfying
$$\norm{\Theta_{ki*}((\Theta_{k*}-\Gamma_{k*})(\om_\psi))}=|\la\Theta_{ki}(T_{ki}),(\Theta_{k*}-\Gamma_{k*})(\om_\psi)\ra|.$$
By Lemma~\ref{l_ineq},
\begin{align*}\norm{\Theta_{ki*}((\Theta_{k*}-\Gamma_{k*})(\om_\psi))}&\leq\norm{T_{ki}}\la\Theta_{ki}(1),|(\Theta_{k*}-\Gamma_{k*})(\om_\psi)|\ra\\
&\leq\la\Theta_{ki}(1),|(\Theta_{k*}-\Gamma_{k*})(\om_\psi)|\ra.
\end{align*}
Hence,
\begin{align*}\sum_{k,i=1}^\infty\norm{\Theta_{ki*}(\Theta_{k*}(\om_\psi))-\Theta_{ki*}(\Gamma_{k*}(\om_\psi))}&\leq\sum_{k=1}^\infty\sum_{i=1}^\infty\la\Theta_{ki}(1),|(\Theta_{k*}-\Gamma_{k*})(\om_\psi)|\ra\\
&=\sum_{k=1}^\infty\norm{\Theta_{k*}(\om_\psi)-\Gamma_{k*}(\om_\psi)}\\
&<\frac{\ep}{2}.
\end{align*}
By the triangle inequality we obtain
$$\ser k\ser i\norm{\Theta_{ki*}(\Theta_{k*}(\om_\psi))-\Gamma_{ki*}(\om_\psi)}<\ep.$$
Recall that $\cl I_{n+1}$ is a coarse-graining  of $(\Theta_k\circ \Theta_{ki})_{k,i}$. Letting $\cl I_\epsilon$ be the result of applying the same coarse-graining to $(\Gamma_{ki})_{k,i}$, the preceding inequality and the triangle inequality then show that $\cl I_\epsilon$ satisfies $P(n+1)$.
\end{proof}

\begin{remark} The proof of Theorem~\ref{th_lopopescu_std} may seem
complicated when compared to Lo and Popescu's intuitive argument for the special case $\A=M_n\otimes 1$. This may be explained by our approximate version of convertibility together with the additional approximation provided by Proposition~\ref{p_lopopescu},
the latter not being required in the type I case.\end{remark}

Using the representation theory of properly infinite von Neumann algebras, we now remove the standardness assumption in the previous theorem. We require the following lemma.

\begin{lemma}\label{l:st} Let $\cl A\subseteq\BH$ be a von Neumann algebra, $\cl B=\cl A'$, $K$ be a Hilbert space, $\widetilde{\cl A} :=\cl A\oten \cl B(K)\oten 1_K$ and  $\widetilde{\cl B}:=\widetilde{\cl A}'=\cl B\oten 1_K\oten\cl B(K)$. Fix $\xi\in K_1$. If $\widetilde{\Xi}\in\locc^r(\widetilde{\cl A})$, then $\cl E\circ\widetilde{\Xi}\circ \iota\in\locc^r(\cl A)$, where 
$$\cl E:\cl B(H\ten K\ten K)\ni \widetilde{T}\mapsto(\id\ten\om_\xi\ten\om_\xi)(\widetilde{T})\in\cl B(H)$$
and $\iota:\cl B(H)\rightarrow \cl B(H\ten K\ten K)$ is the canonical embedding $T\mapsto T\ten 1_K\ten 1_K$. 
\end{lemma}

\begin{proof} Let $\widetilde{\Psi}_k\in \CP^\sigma_{\widetilde{\cl B}}(\cl B(H\ten K\ten K))$ and $\widetilde{\Phi}_k\in \UCP^\sigma_{\widetilde{\cl A}}(\cl B(H\ten K\ten K))$ satisfy $\widetilde{\Xi}= \sum_{k=1}^\infty \widetilde{\Psi}_k\circ\widetilde{\Phi}_k$.
Define $\Psi_k\in \CP^\sigma_{\cl B}(\cl B(H))$ and $\Phi_k\in \UCP^\sigma_{\cl A}(\cl B(H))$ by
\[\Psi_k=\cl E\circ \widetilde{\Psi}_k|_{\BH\ten 1_K\ten 1_K}\circ \iota\qand \Phi_k=\cl E\circ \widetilde{\Phi}_k|_{\BH\ten 1_K\ten 1_K}\circ \iota.\]
Then $\Xi:=\sum_{k=1}^\infty \Psi_k\circ\Phi_k$ is a one-way right LOCC map on $\BH$ relative to
$\cl A$ so it only remains to show that $\Xi=\cl E\circ \tilde \Xi\circ \iota$.
Let $\sigma_{r,s}$, for $r,s\in \{2,3,4,5\}$, be the flip between terms $r$ and $s$
acting on the tensor product $H\otimes K \otimes K \otimes K \otimes K$.
For every $T\in\BH$ and $\rho\in\cl{T}(H)$, we have
\bgroup\allowdisplaybreaks
\begin{align*}
&\duality{\rho}{\Xi(T)}=\sum_{k=1}^\infty\duality{(\Psi_k)_*(\rho)}{\Phi_k(T)}\\
&=\sum_{k=1}^\infty\duality{(\widetilde{\Psi}_k)_*(\rho\ten\om_\xi\ten\om_\xi)}{\cl E(\widetilde{\Phi}_k(T\ten 1\ten 1))\ten 1\ten 1}\\
&=\sum_{k=1}^\infty\duality{(\cl E_*\ten \id\ten \id)((\widetilde{\Psi}_k)_*(\rho\ten\om_\xi\ten\om_\xi))}{\widetilde{\Phi}_k(T\ten 1\ten 1)\ten 1\ten 1}\\
&=\sum_{k=1}^\infty\duality{\sigma_{24}\sigma_{35}((\widetilde{\Psi}_k)_*(\rho\ten\om_\xi\ten\om_\xi)\ten\om_\xi\ten\om_\xi)}{\widetilde{\Phi}_k(T\ten 1\ten 1)\ten 1\ten 1}\\
&=\sum_{k=1}^\infty\duality{\sigma_{35}((\widetilde{\Psi}_k)_*(\rho\ten\om_\xi\ten\om_\xi)\ten\om_\xi\ten\om_\xi)}{\widetilde{\Phi}_k(T\ten 1\ten 1)\ten 1\ten 1}\\
&=\sum_{k=1}^\infty\duality{(\widetilde{\Psi}_k)_*(\rho\ten\om_\xi\ten\om_\xi)\ten\om_\xi\ten\om_\xi}{\widetilde{\Phi}_k(T\ten 1\ten 1)\ten 1\ten 1}\\
&=\duality{\rho\ten\om_\xi\ten\om_\xi\ten\om_\xi\ten\om_\xi}{\widetilde{\Xi}(T\ten 1\ten 1)\ten 1\ten 1}\\
&=\duality{\rho\ten\om_\xi\ten\om_\xi}{\widetilde{\Xi}(T\ten 1\ten 1)}\\
&=\duality{\rho}{\cl E\circ\widetilde{\Xi}\circ\iota(T)},
\end{align*}
\egroup
where in the fifth equality we used the fact that $\widetilde{\Phi}_k(T\ten 1\ten 1)\ten 1\ten 1\in \cl B(H)\ten 1\ten \cl B(K)\ten 1\ten 1$ is symmetric under $\sigma_{24}$ and in the sixth equality we used the fact that $(\widetilde{\Psi}_k)_*$ acts trivially on the third leg. (These facts follow from Proposition~\ref{p_comp}, for example.)
\end{proof}

Recall that a von Neumann algebra is $\sigma$-finite if every set of mutually orthogonal projections is at most countable, and that every von Neumann algebra $\A\subseteq\BH$ on a separable Hilbert space~$H$ enjoys this property.

\begin{theorem}\label{t:lopopescu} Let $\cl A\subseteq\BH$ be a semi-finite
factor on a separable Hilbert space~$H$.
Given $\Theta\in\locc(\cl A)$ and $\psi\in H$, for every $\ep>0$ there exists
$\Xi \in \LOCC^r(\A)$ such that
$$\nor{\Theta_*(\om_\psi)-\Xi_*(\om_\psi)} < \ep.$$
\end{theorem}

\begin{proof}
Clearly, we may assume that $\psi\in H_1$.
Let $K$ be a separable infinite-dimensional Hilbert space, and
consider the factors $\widetilde{\cl A} :=\cl A\oten \cl B(K)\oten 1_K$ and $\widetilde{\cl B}:=\widetilde{\cl A}'=\cl B\oten 1_K\oten\cl B(K)$,
acting on $H\ten K\ten K$, and equip $\widetilde{\cl A}$ with the trace $\widetilde{\tau} = \tau\otimes {\rm tr} \otimes {\rm tr}$.
Letting $L^{\infty}(\widetilde{\cl A})\subseteq\cl B(L^2(\widetilde{\cl A}))$ denote the standard representation of $\widetilde{\cl A}$, one sees that both $\widetilde{\cl B}=\widetilde{\cl A}'$ and $L^{\infty}(\widetilde{\cl A})'\cong L^{\infty}(\widetilde{\cl A})$ are $\sigma$-finite and properly infinite factors. Hence, by \cite[Proposition V.3.1]{t1},
the representations $(\widetilde{\cl A},H\ten K\ten K)$ and $(\widetilde{\cl A},L^2(\widetilde{\cl A}))$ are unitarily equivalent,
implemented by the unitary operator
$U : H\ten K \ten K\rightarrow L^2(\widetilde{\cl A})$, say.

Clearly, $\Ad(U)\circ (\Theta\ten\id_{\cl B(K\ten K)})\circ \Ad(U^*)\in\locc(L^\infty(\widetilde{\cl A}))$. Since~$L^\infty(\widetilde{\A})$ is a factor, by Theorem~\ref{th_lopopescu_std}, for every $\ep>0$ there exists a one-way right LOCC map $\widetilde{\Xi}:\cl B(L^2(\widetilde{\cl A}))\rightarrow\cl B(L^2(\widetilde{\cl A}))$ relative to $L^\infty(\widetilde{\cl A})$ such that
$$\norm{(\Ad(U)\circ (\Theta\ten\id_{\cl B(K\ten K)})\circ \Ad(U^*))_*(\om_{U(\psi\ten\xi\ten\xi)})-\widetilde{\Xi}_*(\om_{U(\psi\ten\xi\ten\xi)})}<\ep.$$
Then $\Ad(U^*)\circ\widetilde{\Xi}\circ\Ad(U):\cl B(H\ten K\ten K)\rightarrow\cl B(H\ten K\ten K)$ is a one-way right LOCC map relative to $\widetilde{\cl A}$ satisfying
$$\norm{(\Theta\ten\id_{\cl B(K\ten K)})_*(\om_{\psi\ten\xi\ten\xi})-(\Ad(U^*)\circ\widetilde{\Xi}\circ\Ad(U))_*(\om_{\psi\ten\xi\ten\xi})}<\ep.$$
Fix a vector $\xi\in K_1$. By Lemma~\ref{l:st}, the map
$$\Xi:=\cl E\circ(\Ad(U^*)\circ\widetilde{\Xi}\circ\Ad(U))\circ\iota\in\locc^r(\cl A)$$
where $\cl E$ and $\iota$ are defined as in the Lemma. It follows that
\begin{align*}
&\norm{\Theta_*(\om_\psi)-\Xi_*(\om_\psi)}\\
&=\norm{\iota_*\circ(\Theta\ten\id_{\cl B(K\ten K)})_*\circ\cl E_*(\om_\psi)-\iota_*\circ(\Ad(U^*)\circ\widetilde{\Xi}\circ\Ad(U))\circ\cl E_*(\om_\psi)}\\
&\le \norm{(\Theta\ten\id_{\cl B(K\ten K)})_*(\om_{\psi\ten\xi\ten\xi})-(\Ad(U^*)\circ\widetilde{\Xi}\circ\Ad(U))_*(\om_{\psi\ten\xi\ten\xi})}<\ep.\;\qedhere\end{align*}
\end{proof}

By left-right symmetry, Theorem~\ref{t:lopopescu} immediately yields the following corollary:

\begin{corollary}\label{c_oneway}
Let $\cl A\subseteq\BH$ be a semi-finite factor on a separable Hilbert space.
For unit vectors $\psi,\vphi\in H$, the following are equivalent:
\begin{enumerate}[(i)]
\item  $\psi$ is approximately convertible to $\vphi$ via
  $\locc(\cl A)$;
\item  $\psi$ is approximately convertible to $\vphi$ via $\locc^r(\cl A)$;
\item $\psi$ is approximately convertible to $\vphi$ via $\locc^l(\cl A)$.
\end{enumerate}
\end{corollary}

\section{The Main Theorem}\label{s:main}

In this section we establish Theorem~\ref{th_maj}, a version of Nielsen's theorem for bipartite systems modelled by semi-finite, $\sigma$-finite von Neumann algebras (or by standardly represented von Neumann algebras).
The next group of lemmas will help justify certain technical arguments in its proof.

\begin{lemma}\label{l_M}
  Let $(\cl A,\tau)$ be a semi-finite von Neumann algebra. For any
  $\rho_1,\dots,\rho_n\in \A_*^+$, there exist
  $M_1,\dots,M_{n+1}\in \cl A$ with $\sum_{i=1}^{n+1}M_i^*M_i=1$ such
  that, for $\rho=\sum_{i=1}^n \rho_i$, we have
  \[M_i \rho M_i^*=\rho_i,\quad 1\le i\le n,\qand M_{n+1}\rho M_{n+1}=0.\]
\end{lemma}
\begin{proof}
  We may assume that $\A$ is standardly represented
  on~$H=L^2(\A,\tau)$, and identify $\A_*$ with $L^1(\A,\tau)$. Let $\B=\A'$. Since $\rho$ is a positive element of
  $L^1(\A,\tau)$, the (positive, densely defined) operator
  $\xi:=\rho^{1/2}$ is an element of $H=L^2(\A,\tau)$. Let $p$
  and $p'$ be the orthogonal projections onto $\overline{\B\xi}$ and
  $\overline{\A\xi}$, respectively, and note that
  \[ p\in \A,\quad p'\in \B,\quad Jp'J=p\qand JpJ=p'\] (see
  \cite[Section IX.1]{t2};
  for the last two equalities, $J\xi=\xi^*=\xi$ so
  $JpJa\xi=Jp(JaJ\xi)=J(JaJ\xi)=a\xi$, so $p'\le JpJ$ and similarly
  $p\le Jp'J$.)

  For $1\le i\le n$, we have $\rho_i\leq\rho$, that is,
  $\om_{\rho_i^{1/2}}\leq \om_\xi$.  By the Radon--Nikodym theorem
  (see~\cite[Proposition~7.3.5]{kr2} and its proof) there
  exists $b_i\in \cl B^+$ such that
  \[\om_{\rho_i^{1/2}}(a) = \ip{a b_i\xi}{\xi} = \ip{a b_i^{1/2}\xi}{b_i^{1/2}\xi} = \om_{b_i^{1/2} \xi}(a),\quad a\in \cl{A}.\]
  By the uniqueness of the GNS representation, there exists a partial
  isometry $v_i\in \cl B$ such that for $c_i:=v_ib_i^{1/2}\in \B$, we have
  \begin{equation}\label{eq_cixi}
  c_i \xi=\rho_i^{1/2},\quad 1\le i\le n.
  \end{equation}
  Consider $M_1,\dots,M_{n+1}\in \A$, given by
  \[ M_i:=Jc_iJp,\quad 1\le i\le n,\quad M_{n+1}:=1-p.\]
  For $1\le i\le n$, using \eqref{eq_cixi} we have
  $$M_i\xi=Jc_iJ\xi=\xi c_i^*=(c_i\xi)^*=(\rho_i^{1/2})^*=\rho_i^{1/2},$$
  so \[ M_i \rho M_i^*=(M_i\xi)(M_i\xi)^*=\rho_i,\quad 1\le i\le n.\]
  Similarly, $M_{n+1}\xi=(1-p)\xi=0$, so $M_{n+1}\rho M_{n+1}^*=0$.
  For $a,b\in \A$, we have
  $p'a\xi=a\xi$ and $p'b\xi=b\xi$; hence, for $1\le i\le n$,
  \begin{align*}
    \ip{JM_i^*M_iJa\xi}{b\xi}&=\ip{JpJc_i^*c_iJpJa\xi}{b\xi}= \ip{c_ip'a\xi}{c_ip'b\xi}
    \\&=\ip{c_ia\xi}{c_ib\xi}=\ip{ac_i\xi}{bc_i\xi}=\ip{a\rho_i^{1/2}}{b\rho_i^{1/2}}
    \\&=\tau(\rho_i^{1/2}b^*a\rho_i^{1/2})=\tau(b^*a\rho_i).
  \end{align*}
  Thus, the operator $S:=\sum_{i=1}^n M_i^*M_i$ satisfies $ \ip{JSJ a\xi}{b\xi}=\tau(b^*a\rho)=\ip{a\xi}{b\xi}$.
  Hence, $p'JSJp'=p'$, so
  $$S = pSp = Jp'JSJp'J = Jp'J = p,$$ and
  \[ \sum_{i=1}^{n+1}M_i^*M_i=S+M_{n+1}^*M_{n+1}=p+(1-p)=1.\qedhere\]
\end{proof}

The version of the following lemma for the case where $\ep = 0$ is well-known. We require the
following approximate extension.



\begin{lemma}\label{l_pure}
Let $H$ be a Hilbert space, $\vphi\in H_1$ and $(\omega_k)_{k=1}^\infty$ be a sequence in $\cl T(H)^+$ such that $\sum_{k=1}^\infty\om_k$ converges weakly to an element in the closed unit ball of $\cl T(H)$. Set $\alpha_k=\langle \phi\phi^*,\om_k\rangle$, $k\in \bb{N}$.
If $\epsilon > 0$ and
\begin{equation}\label{l_1}\nor{\om_\vphi - \sum_{k=1}^\infty\om_k}<\ep,\end{equation} then
$\sum_{k=1}^\infty\nor{\om_k-\alpha_k\om_\vphi} < 2\sqrt{\ep}+\ep$.
\end{lemma}

\begin{proof}
Let $\tr$ denote the canonical trace on $\cl T(H)$ and, for simplicity, write $p = \vphi\vphi^*$ and $p^{\perp} = 1 - p$.
Observe that $p^\perp \om_\vphi p^\perp =0$. Thus, by \eqref{l_1},
\begin{align}\label{eq_tlines}
\sum_{k=1}^\infty\nor{p^\perp\om_k p^\perp}&=\sum_{k=1}^\infty\tr(p^{\perp}\om_kp^{\perp})=\tr\bigg(p^{\perp}\bigg(\sum_{k=1}^\infty\om_k\bigg)p^{\perp}\bigg)\\
&=\nor{p^{\perp}\bigg(\sum_{k=1}^\infty\om_k\bigg)p^{\perp}} < \ep. \nonumber
\end{align}
By the Cauchy--Schwarz inequality, for any $T\in\BH$, we have
\begin{align*}
\left|p\omega_k p^{\perp}(T)\right|
= \left|\om_k(p^{\perp}Tp)\right|
& \leq
\om_k\left(p^{\perp}TT^*p^{\perp}\right)^{1/2}\om_k(p)^{1/2}\\
&\leq \nor{T}\nor{p^{\perp}\om_kp^{\perp}}^{1/2}\om_k(p)^{1/2}.\end{align*}
Hence, $\nor{p\om_kp^\perp}\leq\nor{p^{\perp}\om_kp^{\perp}}^{1/2}\om_k(p)^{1/2}$. Applying the
Cauchy--Schwarz inequality once again and using \eqref{eq_tlines}, we obtain
\begin{align*}
\sum_{k=1}^\infty\nor{p\om_kp^\perp}&\leq\sum_{k=1}^\infty \om_k(p)^{1/2} \nor{p^{\perp}\om_kp^{\perp}}^{1/2}\\
&\leq \bigg(\sum_{k=1}^\infty\om_k(p)\bigg)^{1/2} \bigg(\sum_{k=1}^\infty\nor{p^{\perp}\om_kp^{\perp}}\bigg)^{1/2}
< \sqrt{\ep}.
\end{align*}
Since $\nor{p^\perp \om_k p }=\nor{(p\om_k p^\perp)^*}=\nor{p\om_k p^\perp}$, we have $\sum_{k=1}^\infty \nor{p^\perp \om_k p }< \sqrt{\ep}$. Decomposing $\om_k=p\om_kp+p\om_kp^{\perp}+p^\perp\om_k p+p^\perp\om_k p^\perp$, and noting $p\om_k p=\om_k(p)\om_\vphi=\alpha_k \om_\vphi$,
we see that
\begin{align*}
\sum_{k=1}^\infty\nor{\om_k-
\alpha_k\om_\vphi}&\leq\sum_{k=1}^\infty \Big( \nor{p\om_kp^\perp}+\nor{p^\perp \om_k p }+\nor{p^\perp\om_k p^\perp}\Big) < 2\sqrt{\ep}+\ep.\qedhere
\end{align*}
\end{proof}

\goodbreak
We are now in a position to prove the main result of the paper.
It provides a version of Nielsen's theorem for bipartite systems without any explicit (spatial) tensor product structure.

\begin{theorem}\label{th_maj}
Let $\cl A\subseteq\BH$ be a semi-finite factor on a separable Hilbert space~$H$.
For unit vectors $\psi,\vphi\in H$, the following are equivalent:

(i) \ $\psi$ is approximately convertible to $\vphi$ via $\locc(\cl A)$;

(ii) $\rho_{\psi}\prec \rho_{\vphi}$.
\end{theorem}

\begin{proof}
$(i)\Rightarrow(ii)$
Let $\ep > 0$ and $\delta > 0$ be such that
$2(\delta+\sqrt{\delta}) < \tfrac12\ep$. By Corollary~\ref{c_oneway},
there exists a one-way right LOCC map $\Theta$ relative to $\cl A$ such that \begin{equation}\label{eq_nor}\nor{\om_\vphi-\Theta_*(\om_\psi)}<\delta.
\end{equation}
By Proposition~\ref{p_comp}\,(ii), we may write
$$\Theta(x) = \sum_{k=1}^\infty \Phi_k(a_k^*xa_k), \ \ x\in\BH,$$
for some $\Phi_k\in \ucp^\sigma_{\cl A}(\cl B(H))$ and $a_k\in \cl{A}$, $k\in\N$, with
\begin{equation}\label{eq_akak}
\sum_{k=1}^{\infty} a_k^* a_k = 1,
\end{equation}
where the series converge in the weak* topology.
Let
\[ \omega_k=\Phi_{k*}(a_k\omega_\psi a_k^*) = a_k\Phi_{k^*}(\omega_\psi) a_k^*\in \cl T(H)^+;\]
then the series
$\sum_{k=1}^\infty \omega_k$ is weakly convergent to $\Theta_*(\omega_\psi)$.
Let us write
$\alpha_k:=\revduality{\phi\phi^*}{\omega_k}$, $k\in \bb N$.
Since $\revduality{\phi\phi^*}{\omega_\phi}=1$, the bound~\eqref{eq_nor} implies
\begin{equation}\label{eq_alphak}
\sum_{k=1}^\infty \alpha_k=\revduality{\phi\phi^*}{\Theta_*(\omega_\psi)}\in (1-\delta,1].
\end{equation}
By Lemma~\ref{l_pure},
$\sum_{k=1}^\infty\nor{\omega_k -\alpha_k\om_\vphi}<2\sqrt{\delta}+\delta$.
Taking restrictions to $\cl A$, and using the fact that $\Phi_k|_{\cl A}$ coincides with the identity map, we obtain
\begin{equation}\label{eq_no1}
\sum_{k=1}^\infty\nor{a_k\rho_\psi a_k^*-\alpha_k\rho_\vphi}_1 < 2 \sqrt{\delta} + \delta.
\end{equation}

For each $a\in\cl A$, by \eqref{eq_akak} we have
\begin{align*}
\sum_{k=1}^l &\duality{\rho_\psi^{1/2}a_k^*a_k\rho_\psi^{1/2}}{a}
= \sum_{k=1}^l \tau(a\rho_\psi^{1/2}a_k^*a_k\rho_\psi^{1/2})
= \sum_{k=1}^l \tau(a_k^*a_k\rho_\psi^{1/2}a\rho_\psi^{1/2})\\
&= \sum_{k=1}^l \duality{\rho_\psi^{1/2}a\rho_\psi^{1/2}}{a_k^*a_k} \to_{l\to\infty} \duality{\rho_\psi^{1/2}a\rho_\psi^{1/2}}{1}
=\duality{\rho_\psi}{a}.
\end{align*}
Hence, the series $\sum_{k=1}^\infty\rho_\psi^{1/2}a_k^*a_k\rho_\psi^{1/2}$ converges weakly
to $\rho_\psi$.
By Lemma~\ref{l_dini}, the convergence is in norm.  Choose $L\in \bb N$ so that
\begin{equation}\label{eq_norrho}
\nor{\rho_\psi-\sum_{k=1}^L \rho_\psi^{1/2} a_k^*a_k\rho_\psi^{1/2}}_1<\tfrac12\epsilon.
\end{equation}

By the right polar decomposition, there exists a partial isometry
$v_k\in\cl A$ such that $a_k\rho_\psi^{1/2} = (a_k\rho_\psi a_k^*)^{1/2}v_k^*$, $k\in \bb{N}$.
Writing $\alpha_0 = 1 - \sum_{k=1}^\infty\alpha_k$, we have by \eqref{eq_alphak} that $\alpha_0 <\delta$.
Setting $v_0 = 1$ and
using \eqref{eq_no1} and \eqref{eq_norrho}, we see that
\begin{align*}
\nor{\rho_\psi-\sum_{k=0}^L\alpha_kv_k\rho_\vphi v_k^*}_1
  &< \delta+\nor{\sum_{k=1}^L \rho_\psi^{1/2}a_k^*a_k \rho_\psi^{1/2}-\alpha_kv_k\rho_\phi v_k^*}_1+\tfrac12\epsilon\\
  &= \delta+\nor{\sum_{k=1}^L v_k(a_k \rho_\psi a_k^*-\alpha_k\rho_\phi) v_k^*}_1+\tfrac12\epsilon\\
  &\le \delta+\sum_{k=1}^L \nor{a_k\rho_\psi a_k^*-\alpha_k\rho_\phi}_1+\tfrac12\epsilon
    \\ & < 2(\delta+\sqrt \delta)+\tfrac12\epsilon<\ep.
\end{align*}
Since $\ep>0$ was arbitrary, it follows from \cite[Theorem 2.5\,(3)]{h} 
that $\rho_\psi\prec\rho_\vphi$. 

\smallskip

$(ii)\Rightarrow(i)$
Suppose $\rho_\psi \prec \rho_\vphi$, and fix $\ep>0$. Pick $\delta>0$ such that $4\sqrt{\delta}<\ep$.
Since $\A$ is a factor, by \cite[Theorem 2.5]{h},
there exist a family $(u_i)_{i=1}^n$ of unitary operators in $\cl{A}$ and
a probability distribution $(p_i)_{i=1}^n$,
such that, if $\widetilde{\rho_\psi}=\sum_{i=1}^n p_i u_i\rho_\vphi u_i^*$,
then
$\norm{\rho_\psi-\widetilde{\rho_\psi}}_1 < \delta$.
Set $m=n+1$. By
Lemma~\ref{l_M},
there exist
$M_1,\dots,M_m\in \cl A$ with $\sum_{i=1}^m M_i^*M_i=1$, such that
\begin{equation}\label{e_M}
M_i\widetilde{\rho_\psi}M_i^*= p_i\rho_\vphi\text{ for $1\le i\le n$,}\qand M_m\widetilde{\rho_\psi}M_m^*= 0.
\end{equation}

Let $e_1,\dots,e_m$ be the standard basis of $\bb C^m$,
and consider the UCP maps $\Psi,\Phi:\cl B(H\ten \C^m)\rightarrow\cl B(H)$ given by
\[\Psi(T) = \sum_{i=1}^mM_i^*(\id\ten\om_{e_i})(T)M_i\qand \Phi(T)=\sum_{i=1}^n p_i(\id \ten\om_{e_i})(T)\]
for $T\in \cl B(H\ten \C^m)$.
We have that
\begin{equation}\label{e_pre}
\Psi_*(\rho)=\sum_{i=1}^m M_i\rho M_i^*\ten e_ie_i^*\qand
\Phi_*(\rho)=\sum_{i=1}^n p_i\rho\ten e_ie_i^*,
\end{equation}
for $\rho\in \cl T(H)$. Letting $V,W: H\rightarrow H\ten\C^m\ten\C^m$ be the isometries given by
\[V\eta = \sum_{i=1}^mM_i\eta\ten e_i\ten e_i\qand W\eta=\sum_{i=1}^n \sqrt p_i\eta\ten e_i\ten e_i,\quad\eta\in H,\]
we have Stinespring representations
\begin{equation}
\label{st_VW}\Psi(T)=V^*(T\ten 1)V
\qand \Phi(T) = W^*(T\ten 1)W,\quad T\in \cl B(H\ten \C^m).
\end{equation}

Consider the states $\nu_\psi,\nu_\phi\colon \cl A\ten M_m\to \bb C$ given by
\[ \nu_\psi=\omega_\psi\circ \Psi|_{\cl A\ten
    M_m}\qand\nu_\phi=\omega_\phi\circ \Phi|_{\cl A\ten M_m}.\]
By~\eqref{e_pre} and~\eqref{e_M}, we have
\begin{align*}
\nor{\nu_\psi-\nu_\phi}_{cb}&=\nor{\nu_\psi-\nu_\phi}\\
&=\nor{(\Psi|_{\cl A\ten M_m})_*(\rho_\psi)-(\Phi|_{\cl A\ten M_m})_*(\rho_\vphi)}\\
&=\nor{\sum_{i=1}^mM_i\rho_\psi M_i^* \ten e_ie_i^*- M_i\widetilde{\rho_\psi} M_i^*\ten e_ie_i^*}\\
&=\nor{(\Psi|_{\cl A\ten M_m})_*(\rho_\psi-\widetilde{\rho_\psi})} < \delta.
\end{align*}
Let
$\theta\colon \cl A\otimes M_m\to \cl B (H\otimes \bb C^m\otimes
\bb C^m)$ be the $*$-homomorphism given by $\theta(X)=X\otimes 1$, $X\in \cl A\otimes M_m$.
By \eqref{st_VW}, the maps $\nu_\psi$ and $\nu_\phi$ have Stinespring
representations
\[ \nu_\psi=\omega_{V\psi}\circ \theta\qand \nu_\phi=\omega_{W\phi}\circ \theta.\]
By the continuity of the Stinespring representation \cite[Theorem 1]{ksw} there exist a Hilbert space $K$, a $*$-homomorphism $\pi:\cl A\ten M_m\rightarrow \cl B(K)$, and vectors $\eta_1,\eta_2\in K$ yielding Stinespring representations
\[
\nu_\psi=\om_{\eta_1}\circ \pi\qand
\nu_\phi=\om_{\eta_2}\circ\pi
\] with
\begin{equation}\label{eq_sqrd}
\nor{\eta_1-\eta_2} < \sqrt{\delta}.
\end{equation}
By the uniqueness of Stinespring representations,
there exist partial isometries $U_1:H\ten\C^m\ten\C^m\rightarrow K$ and $U_2:K\rightarrow H\ten\C^m\ten\C^m$
satisfying $U_1V\psi=\eta_1$, $U_2\eta_2=W\vphi$,
\[U_1(X\ten 1)=\pi(X)U_1 \quad \mbox{ and } \quad U_2\pi(X) = (X\ten 1)U_2,\quad X\in \cl A\ten M_m.\]
Let $\cl B=\cl A'$. The preceding relations imply that the contraction $U:=U_2U_1$ satisfies
\[ U\in (\cl A\ten M_m\ten 1)'=\cl A'\ten 1\ten M_m=\cl B\ten 1\ten M_m;\]
moreover, by \eqref{eq_sqrd},
\begin{equation}\label{eq_uvw}
\nor{UV\psi-W\vphi} < \sqrt{\delta}.
\end{equation}
Since $U \in \cl B\ten 1\ten M_m$, 
we have $U = \sum_{k,l=1}^mb_{kl}\otimes 1\otimes e_ke_l^*$ for some $b_{kl}\in\cl B$.
Set $b_i = b_{ii}$, $1\leq i\leq m$.
Then $b_i$ is a contraction in $\cl B$. Let $\Phi_i\in\ucp_{\cl A}^\sigma(\BH)$ be the channel given by
$$\Phi_i(x) = b_i^*xb_i + (1 - b_i^*b_i)^{1/2}x(1 - b_i^*b_i)^{1/2}, \ \ \ x\in\BH,$$
and define $\Theta\in\locc(\cl A)$ by
$$\Theta(x) = \sum_{i=1}^m \Phi_i(M_i^*x M_i), \ \ \ x\in\BH.$$

We claim that $\nor{\Theta_*(\om_\psi)-\om_\vphi} < \ep$, which will complete the proof. To see this, let $P : \C^m\ten\C^m\rightarrow\C^m\ten\C^m$ denote the orthogonal projection onto $\spn\{e_i\ten e_i\mid 1\leq i\leq m\}$, and consider the contraction
\[\tilde{U} = (1\ten P)U \in \cl B(H\otimes \bb C^m\otimes \bb C^m).\]
A calculation shows that
\begin{equation}\label{tUV}
\tilde U V\psi = \sum_{i=1}^m b_i M_i\psi\otimes e_i\otimes e_i.
\end{equation}
Since
$W\phi$ lies in the range of $1\ten P$, the bound~\eqref{eq_uvw} implies
\begin{equation}\label{eq_wdelta}
\nor{\tilde UV\psi-W\vphi} = \nor{(1\ten P)\left(UV\psi-W\vphi\right)} < \sqrt{\delta},
\end{equation}
and so
\begin{align}\label{eq_anewo}
  \nor{\omega_{\tilde U V\psi}-\omega_{W\phi}}\le 2\nor{\tilde U V\psi-W\phi}< 2 \sqrt{\delta}.
\end{align}
Observe that for $x\in \BH$, equation~\eqref{tUV} yields
\begin{align*}
  \duality{\left(\sum_{i=1}^m\omega_{b_iM_i\psi}\right)-\omega_\phi}{x}=\duality{\omega_{\tilde U V\psi}-\omega_{W\phi}}{x\ten1\ten1}
\end{align*}
so, in particular, using \eqref{eq_anewo}, we have
\begin{align}\label{eq_mipsi}
  \nor{\sum_{i=1}^m b_iM_i\omega_\psi M_i^*b_i^*-\omega_\phi} &= \nor{\left(\sum_{i=1}^m \omega_{b_iM_i\psi}\right)-\omega_{\phi}}
  \\&\leq \nor{\omega_{\tilde UV\psi}-\omega_{W\phi}}<2\sqrt\delta. \nonumber
\end{align}

Since $V$ and $W$ are isometries, we have $\|V\psi\|=1=\|W\phi\|$.
Using \eqref{eq_wdelta}, we thus have
$$\|V\psi\| - \|\tilde{U}V\psi\| = \|W\vphi\| - \|\tilde{U}V\psi\| < \sqrt{\delta}.$$
By Lemma~\ref{l_piso},
\begin{align*}
\sum_{i=1}^m\nor{(1-b_i^*b_i)^{1/2}M_i\psi}^2 &=
\sum_{i=1}^m \|M_i\psi\|^2-\|b_iM_i\psi\|^2
  =\|V\psi\|^2-\|\tilde UV\psi\|^2\\
&=\nor{(1-\tilde{U}^*\tilde{U})^{1/2}V\psi}^2 < 2\sqrt{\delta}.
\end{align*}
Thus, using \eqref{eq_mipsi}, we have
\begin{align*}
& \nor{\Theta_*(\om_\psi)-\om_\vphi}\\
&= \nor{\sum_{i=1}^mb_iM_i\om_\psi M_i^*b_i^*+(1-b_i^*b_i)^{1/2}M_i\om_\psi M_i^*(1-b_i^*b_i)^{1/2}-\om_\vphi}\\
&\leq \nor{\sum_{i=1}^mb_iM_i\om_\psi M_i^*b_i^*-\om_\vphi}
+\sum_{i=1}^m\nor{(1-b_i^*b_i)^{1/2}M_i\om_\psi M_i^*(1-b_i^*b_i)^{1/2}}\\
&<2\sqrt\delta+\sum_{i=1}^m\nor{(1-b_i^*b_i)^{1/2}M_i\psi}^2
< 4\sqrt{\delta} < \ep.\qedhere
\end{align*}
\end{proof}


\begin{remark}
The structure theory of type $\mathrm{III}_1$-factors renders state convertibility trivial in that setting. Indeed, if $\cl A$ is a factor of type $\mathrm{III}_1$ with separable predual, then
by~\cite[Theorem XII.5.12]{t2}, for all normal states $\om_1,\om_2$ on $\cl A$, we have
$$\inf\{\norm{u^*\om_1 u - \om_2}\mid u\in \cl{U}(\cl A)\}=0.$$
Hence, given states $\psi,\vphi$ in the representation space $H$ of $\cl A$, for every $\ep>0$ there exists a unitary $u\in\cl A$ such that
$\norm{u^*\om_\psi|_{\cl A} u - \om_\vphi|_{\cl A}}<\ep$. Appealing to continuity and uniqueness of Stinespring representations (as in Proposition~\ref{p_lopopescu}) one can build a channel $\Theta\in\locc(\cl A)$ for which
\[\norm{\Theta_*(\om_\psi)-\om_\vphi} < 2\sqrt{\ep}.\]
Hence, $\psi$ is approximately convertible to $\vphi$ via $\locc(\cl A)$ and vice-versa. The problem of convertibility for general type $\mathrm{III}$ factors remains an interesting open question.
\end{remark}

\begin{remark}
It is natural to ask if the statement of Theorem~\ref{th_maj} holds in the case of general semi-finite von Neumann algebras.
Such an extension would require a treatment of integral decompositions of normal completely positive maps, and
is left for a further study.
Here we only include an illustration involving a typical non-factor case.
Let $\cl D$ be a maximal abelian selfadjoint algebra with separable predual, acting on a Hilbert space $H$.
We may assume, without loss of generality, that $(X,\mu)$ is a probability measure space such that $H = L^2(X,\mu)$
and $\cl D = \{M_a : a\in L^{\infty}(X,\mu)\}$, where, for $a\in L^{\infty}(X,\mu)$, we have let $M_a\in \cl B(H)$ be the operator
of multiplication by $a$. We equip $\cl D$ with the trace $\tau$ given by $\tau(M_a) = \int_{X} a\, d\mu$.
Note that, since $\cl D = \cl D'$, we have $\locc^r(\cl D)=\locc^l(\cl D)=\locc(\cl D)$, and these sets consist of all unital positive Schur multipliers relative to $(X,\mu)$, that is, the maps $\Phi : \cl B(H)\to \cl B(H)$ of the form
\begin{equation}\label{eq_schur}
\Phi(T) = \sum_{i=1}^{\infty} M_{a_i}^* T M_{a_i}, \ \ \ T\in \cl B(H),
\end{equation}
where $a_i\in L^\infty(X,\mu)$, $i\in \bb N$ and
$$\sum_{i=1}^{\infty} |a_i(s)|^2 = 1 \ \ \mbox{ for almost all } s\in X.$$
Let $\psi,\nph\in H$. We claim that the following are equivalent:
\begin{itemize}
\item[(i)] there exists $\Phi\in \locc(\cl D)$ such that $\Phi_*(\om_\psi) = \om_\nph$;

\item[(ii)] $\psi$ is approximately convertible to $\nph$ via $\LOCC(\D)$;

\item[(iii)] $|\psi| = |\nph|$ almost everywhere.
\end{itemize}

Indeed, the implication (i)$\Rightarrow$(ii) is trivial.
Assuming (ii), fix $\epsilon > 0$ and let
$\Phi\in \locc(\cl D)$ be such that $\|\omega_{\nph} - \omega_{\psi} \circ \Phi\| < \epsilon$.
Writing $\Phi$ in the form \eqref{eq_schur}, we have
\begin{eqnarray*}
\sup_{\|c\|_{\infty} \leq 1} \left| \int_X c (|\nph|^2 - |\psi|^2) \,d\mu \right|
& = &
\sup_{\|c\|_{\infty} \leq 1} \left| \int_X c (|\nph|^2 - \left(\sum_{i=1}^{\infty}|a_i|^2\right) |\psi|^2) \,d\mu \right|\\
& = &
\sup_{\|c\|_{\infty} \leq 1} |\om_\nph(M_c) - \om_\psi(\Phi(M_c))|\\
& \leq &
\|\om_\nph - \om_\psi \circ \Phi\| < \epsilon.
\end{eqnarray*}
Thus, $\||\nph|^2 - |\psi|^2\|_1 < \epsilon$. Hence $|\nph|^2 = |\psi|^2$ in $L^1(X,\mu)$, and (iii) follows.
Finally, assuming (iii), let $\theta : X\to\bb{C}$ be a unimodular function such that $\nph = \theta \psi$, and
let $\Phi : \cl B(H) \to \cl B(H)$ be the map given by $\Phi(T) = M_{\theta}^* T M_{\theta}$.
Then $\Phi\in \locc(\cl D)$ and $\Phi_*(\om_\psi) = \om_\nph$.
\end{remark}

\section{Trace Vectors and Entanglement in $\TO$-factors}\label{s_ttf}

This section is dedicated to some examples and applications of our convertibility result from Section~\ref{s:main}.
In its first part, we consider a generalisation of maximally entangled vectors to the commuting von Neumann algebra setting,
while in its second part we show that entropy of states, relative to the trace, is an entanglement monotone
in the sense of \cite{v}.

\subsection{Trace vectors}
\begin{definition} Let $\cl A\subseteq\BH$ be a finite factor on a Hilbert space. A unit vector $\psi\in H$ is said to be a \textit{trace vector} for $\cl A$ if $\om_{\psi}|_{\cl A}=\tau$, the unique (normal) tracial state on $\cl A$.
\end{definition}

\begin{remark}\label{r_1A}
  Since $\omega_{\psi}|_{\A}(a)=\tau(\rho_\psi a)$ for $a\in \A$, we
  see that $\psi$ is a trace vector if and only if $\rho_\psi=1_\A$.
\end{remark}

It follows from Nielsen's theorem \cite{nielsen} that the maximally entangled state $\psi=\frac{1}{\sqrt{n}}\sum_{i=1}^ne_n\ten e_n\in\C^n\ten\C^n$ is LOCC-convertible (that is, convertible via $\locc(M_n\otimes 1)$) to any other state $\vphi\in\C^n\ten\C^n$. Notice that $\om_\psi|_{M_n\ten 1}=\frac{1}{n}\tr$, the normalised trace on $M_n$. Hence, $\psi$ is a trace vector for $M_n\ten 1_n\subseteq \cl B(\C^n\ten \C^n)$. The next proposition shows that trace vectors play the role of maximally entangled states relative to $\TO$-factors, and provides additional evidence for viewing maximal entanglement through the lens of tracial states \cite[\S V.A]{keylsw}.



\begin{proposition}\label{II1}
Let $\cl A\subseteq\BH$ be a $\TO$-factor on a separable Hilbert space~$H$.
If $\psi\in H$ is a trace vector for $\cl A$, then $\psi$ is approximately convertible to $\vphi$ via $\locc(\cl A)$ for any $\vphi\in H_1$. Conversely, if there exists a trace vector $\psi_0\in H$ for $\cl A$, and $\psi\in H_1$ is approximately convertible to any $\vphi\in H_1$ via $\locc(\cl A)$,
then $\psi$ is a trace vector for $\cl A$.
\end{proposition}

\begin{proof}
Suppose that $\psi$ is a trace vector for $\cl A$.
By Remark~\ref{r_1A}, $\rho_\psi = 1_\A$.
The map on $\A$, given by $a\mapsto \tau(a)1_\A$, is doubly stochastic
(i.e., it is positive, normal, unital and trace-preserving) and its extension to $L^1(\A,\tau)$ maps $\rho_\phi$ to $1_\A=\rho_\psi$, since  $\tau(\rho_\phi)=\omega_\phi(1_\A)=(\phi,\phi)=1$. It follows from \cite[Theorem~4.5]{h} that $\rho_\psi\prec\rho_\vphi$, hence, $\psi$ is approximately convertible to $\vphi$ via $\locc(\cl A)$ by Theorem~\ref{th_maj}.



For the converse statement, suppose $\psi_0\in H_1$ is a trace vector for $\cl A$, and that $\psi\in H_1$ is approximately convertible to every $\vphi\in H_1$ via $\locc(\cl A)$. Then $\psi$ is approximately convertible to $\psi_0$ and vice-versa.
By Theorem~\ref{th_maj} and Remark~\ref{r_1A},
$\rho_\psi\prec 1_{\cl A}$ and $1_{\cl A} \prec\rho_\psi$, that is, $\rho_\psi$ and $1_{\cl A}$
are spectrally equivalent in the sense of \cite[\S3]{h}.
By \cite[Theorem 3.4\,(2)]{h}, for every $\ep>0$, there exists a unitary $u\in\cl A$ such that
$$\norm{\rho_\psi - 1_{\cl A}}_1 = \norm{\rho_\psi-u\cdot\rho_{\psi_0}\cdot u^*}_1 < \ep.$$
Since $\ep > 0$ was arbitrary we have $\rho_\psi = 1_{\cl A}$; by Remark~\ref{r_1A}, $\psi$ is a trace vector for $\cl A$.
\end{proof}

Amongst $\TO$-factors, the hyperfinite (i.e., approximately finite dimensional) $\TO$-factor is best suited for applications in mathematical physics. In that context, it typically appears through an infinite tensor product construction, an algebra of canonical commutation/anti-commutation relations, or an irrational rotation algebra. We now present examples of maximally entangled states relative to the hyperfinite $\TO$-factor in each of the three aforementioned manifestations.

\begin{example} This example is based on \cite[\S4.2]{kmsw}. Consider an infinite spin chain consisting of infinitely many qubits arranged on a one-dimensional lattice, say $\Z$. The underlying $C^*$-algebra of the system is given by the infinite tensor product $A=\bigotimes_{\Z} M_2$, that is, the inductive limit of the system $A_F=\bigotimes_{n\in F} M_2$, with canonical inclusion maps, where $F$ ranges through the finite subsets of $\Z$. For $n\in\Z$, let $\psi_n$ be the maximally entangled state on $A_{\{-n,n+1\}}$.
Then $\om=\bigotimes_{n\in\Z}\psi_n \psi_n^*$ defines a state on $A$. Let $\cl A=\pi_\om(A_{(-\infty,0)})''\subseteq\cl B(H_\om)$ be the von Neumann algebra generated by the left half-chain in the cyclic GNS-representation $(H_\om,\pi_\om,\psi_\om)$ of $\om$. Then $\cl B:=\cl A'=\pi_\om(A_{[0,\infty)})''$ is the von Neumann algebra generated by the right half-chain. Both $\cl A$ and $\cl B$ are $\TO$-factors, and by construction it follows that $\psi_\om$ is a trace vector for $\cl A$. Thus, by Proposition~\ref{II1}, $\psi_\om$ is approximately convertible to any state $\vphi\in H_\om$ via $\locc(\cl A)$. Naturally, one may view $\psi_\om$ as a state representing infinitely many pairs of entangled qubits.

\end{example}

\begin{example}\label{eg_fock} Let $K$ be a real Hilbert space and $H=K\oplus iK$ its complexification. Let $\cl F_a(H)$ denote the anti-symmetric Fock space over $H$, given by
$$\cl F_a(H)=\bigoplus_{n\geq 0}\wedge^n H,$$
where $\wedge^n H$ is the anti-symmetric subspace of $H^n:=\bigotimes_{k=1}^n H$ for $n\ge1$ and $\wedge^0:=\C$. For $\psi\in H$, let $a(\psi)^*$ and $a(\psi)$ denote the Fock creation and annihilation operators, namely the bounded~\cite{de} linear maps $\cl F_a(H)\to \cl F_a(H)$
given by
$$a(\psi)^*\vphi=\sqrt{n+1}P_a^{n+1}(\psi\ten\vphi), \ \ a(\psi)\vphi=\sqrt{n}P_a^n(\psi^*\ten \id)\vphi,$$
where $n\geq 1$, $\vphi\in\wedge^n H$ and $P_a^n:H^n\rightarrow\wedge^n H$ is the canonical projection.
Let $S\in \B(\cl F_a(H))$ denote the parity operator defined by $S=\bigoplus_{n\geq0}(-1)^{\ten n}$. Letting $B(\psi):=a(\psi)^*+a(\psi)$ represent the corresponding (self-adjoint) Fermionic field operators, it follows that $\cl A:=\{B(\psi)\mid\psi\in K\}''$ is a $\TO$-factor associated to a real-wave representation of the canonical anti-commutation relations~\cite[\S13]{dg} whose commutant satisfies $\cl B:=\cl A'=\{S B(i\psi)\mid \psi\in K\}''$. 

It is known that the vacuum vector $\Om=(1,0,0,\ldots)\in\cl F_a(H)$ is a quasi-free
trace vector for $\cl A$, with 
\begin{equation}\label{e:vac}(B(\psi)B(\vphi)\Om,\Om)=(\psi,\vphi), \ \ \ \psi,\vphi\in K.\end{equation}
More generally, given an anti-symmetric tensor $c\in H\wedge H$ (seen as a Hilbert-Schmidt operator from $\overline{H}$ to $H$), the Fermionic Gaussian vector associated with $c$ is given by
$$\Om_c=\det(1+c^*c)^{-\frac{1}{4}}e^{-\frac{1}{2}a^*(c)}\Om,$$
where $\det(\cdot)$ is the Fredholm determinant and $a^*(c)$ is the two particle creation operator defined by
$$a^*(c)\psi=\sqrt{(n+2)(n+1)}P_a^{n+2}(c\ten\psi),  \ \ \ \psi\in\wedge^n H, \ n\geq 1.$$
Such vectors occur in the Hartree--Fock--Bogoliubov method for approximating Fermionic systems (see e.g.~\cite[\S4]{dmm}), which is related to the Bardeen--Cooper--Schrieffer theory of superconductivity \cite{bcs}. For every $c$ there exists an orthogonal transformation $O_c$ on $K$, and a unitary $U_c$ on $\cl F_a(H)$ satisfying $\Om_c=U_c^*\Om$ and $U_cB(\psi)U_c^*=B(O_c\psi)$, $\psi\in K$ \cite{dg}. Thus, $\om_{\Om_c}|_{\cl A}=\om_{\Om}\circ\Ad(U_c)$, and it follows from \eqref{e:vac} that $\Om_c$ is also a trace vector for $\cl A$. Thus, by Proposition~\ref{II1}, any of the Fermionic Gaussian vectors $\Om_c$ may be converted into any Fock state $\vphi\in\cl F_a(H)$ by means of local operations and classical communication relative to the real-wave representation $\cl A$ of the CAR\@. In particular, the vectors~$\Omega_c$ display properties of maximal entanglement relative to $\cl A$ and its commutant.

\end{example}

We present one more instance of Proposition~\ref{II1}, based on the example from \cite[\S7]{bkk2}, which in turn was partly motivated by \cite{Faddeev}. This example is a particular realization of the irrational rotation algebra and is related to discretised CCR relations, whose relevance to numerical analysis of quantum systems was advocated by Arveson \cite{arv}.

\begin{example} Suppose Alice and Bob have access to a quantum system represented by the Hilbert space $L^2(\R)$. Let $q$ and $p$ denote the self-adjoint  operators corresponding to position and momentum:
$$q\psi(x)=x\psi(x), \ \ \ p\psi(x)=i\frac{d}{dx}\psi(x),$$
where $\psi$ belongs to a common dense domain for $q$ and $p$. Suppose that Alice can measure periodic functions of position and momentum, with periods $t_q$ and $t_p$, respectively. Such functions are given (respectively) by integer powers of the unitary operators
$$U:=e^{i\omega_q q} \qand  V:=e^{i\omega_p p},$$
where, following \cite{bkk2}, we let $\om_q:=\frac{2\pi}{t_q}$ and $\om_p:=\frac{2\pi}{t_p}$. The operators $U$ and $V$ satisfy
$$UV=e^{2\pi i\theta}VU,$$
where $\theta:=\frac{\om_q\om_p}{2\pi}$. In what follows, we assume that $\om_q\om_p>4\pi$ and that $\theta$ is irrational.

The algebra describing Alice's measurement statistics is the von Neumann subalgebra $\cl A$ of $\cl{B}(L^2(\R))$ generated by $U$ and $V$, and is known to be a type $\TO$-factor. The $C^*$-algebra generated by $U$ and $V$ is known as the irrational rotation algebra corresponding to $\theta$. The von Neumann algebra describing Bob's measurement statistics, $\cl B = \cl A'$, is generated by
$$U':=e^{i\frac{\om_q}{\theta}q}, \qand  V':=e^{i\frac{\om_p}{\theta}p},$$
and is also a type $\TO$-factor.

Let $\psi=\frac{1}{\sqrt{2t_q}}\chi_{[-t_q,t_q]}\in L^2(\R)$. If $x\in[-t_q,t_q]$ and $m\in\Z$ then,
since $\om_p>\frac{4\pi}{\om_q}=2t_q$, it follows that $x+m\om_p\in[-t_q,t_q]$ if and only if $m=0$. Hence, for all $n,m\in\Z$ we have
\begin{align*}(U^nV^m\psi,\psi)&=\frac{1}{2t_q}\int_{\R}e^{in\frac{2\pi}{t_q} x}\chi_{[-t_q,t_q]}(x+m\om_p)\chi_{[-t_q,t_q]}(x) \ dx\\
&=\delta_{m,0}\frac{1}{2t_q}\int_{-t_q}^{t_q}e^{in\frac{2\pi}{t_q}x} \ dx\\
&=\delta_{m,0}\delta_{n,0}1.\end{align*}
By \cite[Corollary VI.1.2]{d}, $\om_\psi|_{\cl A}$ is the unique normal tracial state $\tau$ on $\cl A$. By Proposition~\ref{II1}, we have that $\psi$ is approximately convertible to any unit vector $\vphi\in L^2(\R)$ via $\locc(\cl A)$.

One can think of $\psi$ as representing the state of a particle whose position is uniformly distributed over the interval $[-t_q,t_q]$. This uniformity is playing the role of maximal entanglement relative to the bipartite system $(\cl A, \cl B)$.
\end{example}

\subsection{Entanglement monotones}

The practical importance of quantifying the degree of entanglement present in a given state cannot be overestimated. In the standard finite-dimensional tensor product framework, this quantification is studied through notions of entanglement measures. Since entanglement at its very core is a form of non-local quantum correlation, any reasonable entanglement measure ought to be monotonic with respect to local operations and classical communication. The term entanglement monotone has since emerged for such a measure, and it was argued by Vidal~\cite{v} that monotonicity under LOCC is the \textit{only} natural requirement for measures of entanglement. As an application of our main result, we show that the entropy of the singular value distribution satisfies this requirement for pure states relative to $\TO$-factors, thus yielding an entanglement monotone.

Let $\cl A$ be a $\TO$-factor with unique tracial state $\tau$. Given a normal state $\rho\in\cl S(\cl A)$, we define the \textit{entropy of $\rho$ relative to $\tau$} by
$$H_\tau(\rho):=H(\mu(\rho))=-\int_0^1\mu_t(\rho)\log(\mu_t(\rho)) \ dt.$$
By splitting the entropy function $\eta(x)=-x\log(x)$ into $\chi_{[0,1]}\eta+\chi_{(1,\infty)}\eta$, and applying \cite[Remark 3.3]{fk} to the non-negative Borel functions $\chi_{[0,1]}\eta$ and $-\chi_{(1,\infty)}\eta$, it follows that
$$H_\tau(\rho)=\tau(\eta(\rho))=-\tau(\rho\log(\rho))=-S(\rho,\tau),$$
whenever $|H_\tau(\rho)|<\infty$, where $S(\cdot,\cdot)$ is the relative entropy between normal states of $\cl A$ (see e.g.~\cite[\S5]{op}).
 As such, we see that $H_\tau(\rho)\leq 0$ and $H_\tau(\rho)=0$ if and only if $\rho=\tau$. In particular, if $\cl A\subseteq\BH$ and $\psi\in H_1$, then $H_\tau(\rho_\psi)=0$ if and only if $\psi$ is a trace vector for $\cl A$.
 
The fact that $H_\tau(\rho)$ takes negative values is consistent with the differential entropy theory of continuous systems. For example, let $U$ be a non-empty open subset of $[0,1]$ and $f=\chi_U/\lambda(U)\in L^1([0,1],\lambda)$  be the corresponding density with respect to the Lebesgue measure $\lambda$. The entropy of $f$ relative to the probability space $([0,1],\lambda)$ is 
$$H(f)=-\int_0^1 f(x)\log(f(x)) \ d\lambda(x)=\log(\lambda(U))\in(-\infty,0].$$

\begin{proposition}\label{p:monotone} Let $\cl A\subseteq\BH$ be a $\TO$-factor with tracial state $\tau$. The function
$$H_\tau:\cl S(\cl A)\ni \rho\mapsto H_\tau(\rho)\in[-\infty,0]$$
is non-increasing under approximate convertibility by $\locc(\cl A)$, when restricted to states of the form $\rho_\psi$, $\psi\in H_1$.
\end{proposition}

\begin{proof} First note that for any state $\rho\in\cl S(\cl A)$,
$$H_\tau(\rho)=-S(\rho,\tau)=-S(\mu(\rho),\chi_{[0,1]}),$$
where $S(\mu(\rho),\chi_{[0,1]})$ is the relative entropy of the density $\mu(\rho)\colon t\mapsto \mu_t(\rho)$ on $[0,1]$ with respect to the uniform distribution. Given $\psi,\vphi\in H_1$, if $\psi$ is approximately convertible to $\vphi$ via $\locc(\cl A)$ then by Theorem~\ref{th_maj} we have $\rho_{\psi}\prec \rho_{\vphi}$, meaning that $\mu_t(\rho_\psi)\prec\mu_t(\rho_\vphi)$ as probability densities on $[0,1]$. By \cite[Theorem 10]{veh}, it follows that
$$S(\mu_t(\rho_\psi),\chi_{[0,1]})\leq S(\mu_t(\rho_\vphi),\chi_{[0,1]}).$$
Hence, $H_\tau(\rho_\psi)\geq H_\tau(\rho_\vphi)$.
\end{proof}

\begin{remark} In the proof of Proposition~\ref{p:monotone} we could instead appeal to \cite[Theorem 4.7\,(1)]{h} for the connection between majorisation and double stochasticity together with monotonicity of the relative entropy \cite[Theorem 5.3]{op}.
\end{remark}

Recall that any density matrix $\rho\in M_n$ satisfies
$$S(\rho)=-S(\rho,\tau)+\log(n),$$
where $S(\cdot)$ is the von Neumann entropy and $\tau=\frac{1}{n}\tr$ is the maximally mixed state (the $\log(n)$ factor would disappear if we used the unnormalised trace $\tr$). It is known that the restriction of any entanglement monotone to pure states $\psi\in \C^n\ten\C^n$ is a concave function of the reduced density $\rho_\psi=(\id\ten\tr)(|\psi\ra\la\psi|)$ \cite[Theorem 3]{v}. The entanglement monotone $-S(\rho_\psi,\tau)$ is equivalent (up to the translational factor $\log(n)$) to the common choice of $S(\rho_\psi)$, and is the finite-dimensional analogue of our proposed monotone above. Note that $-S(\rho_\psi,\tau)\in[-\log(n),0]$, with the largest value of 0 occurring for maximally entangled $\psi$, and the lowest value of $-\log(n)$ occurring for separable $\psi$.

In the von Neumann algebraic formulation of quantum field theory, a bipartite system is modelled by a pair $\cl A,\cl B\subseteq\BH$ of von Neumann algebras such that $\cl A\subseteq\cl B'$ (not necessarily equal). Calculating reduced von Neumann entropies in this context is problematic as entropies are always infinite in non-type I systems \cite[Theorem 6.10]{op}. This can be circumnavigated when $(\cl A,\cl B)$ satisfies the so-called split property, meaning there exists a type I factor $\cl M$ satisfying $\cl A\subseteq\cl M\subseteq\cl B'$. In the case of $\TO$-factors, our monotone $-S(\rho,\tau)$ has the advantage that it is ``typically'' finite and can produce meaningful entropic measures without additional restrictions on $(\cl A,\cl B)$. Even in type III settings, it suggests that the relative entropy with respect to a fixed reference state (for which there is often a canonical choice) could be a novel way to measure entanglement.


\section{Outlook}

Several natural lines of investigation arise from this work. First, we intend to study the generalisation of our main result to the non-factor setting in connection with \cite{h2}, as mentioned at the end of Section~\ref{s:main}.
This could be useful for the study of entanglement in hybrid systems \cite{kuper,DevShor,bkk0,bkk1}.
Second, a rectangular version of our main theorem,
describing convertibility between states $\psi\in H$ and $\vphi\in K$, with respect to distinct bipartite systems
$(\cl A_1,\cl B_1)$ in $\BH$ and $(\cl A_2,\cl B_2)$ in $\cl B(K)$, would be desirable.
Among other things, this could have applications to the structure of quantum correlation
matrices and values of certain non-local games \cite{psstw}.
We also plan to explore notions of distillability/dilution of entanglement for general bipartite systems in connection with this and previous work \cite{vw,kmsw}. In that direction it would be interesting to explore uniqueness of entanglement monotones in the asymptotic regime for pure states, analogous to the finite-dimensional setting \cite{pr}. Finally, we observe that while many of the von Neumann algebras of interest in algebraic quantum field theory are of type III \cite{haag}, and hence not semi-finite,  it is conceivable that semi-finite von Neumann algebras may be relevant in the discretisation of space-time via tensor networks---a current area of research---analogous to those arising from discretisations of the canonical commutation relations \cite{slawny}.

\vspace{0.1in}

{\noindent}{\it Acknowledgements.} The authors are grateful for the reviewers' comments, which improved the overall presentation of the paper. J.C. was partially supported by NSERC Discovery Grant RGPIN-2017-06275. D.W.K. was partly supported by NSERC Discovery Grant 400160 and by a Royal Society grant allowing research visits to Queen's University Belfast.
R.H.L.~was partly supported by UCD Seed Funding.


\begin{thebibliography}{99}

\bibitem{arv}
{\sc W. Arveson},
{\it Discretized CCR algebras},
{\rm J. Operator Thy. 26 (1991), no. 2, 225--239}.

\bibitem{bcs}
{\sc J. Bardeen, L. N. Cooper and J. R. Schrieffer},
{\it Theory of superconductivity},
{\rm Phys. Rev. 108 (1957), 1175}.

\bibitem{bkk0}
{\sc C. B\'{e}ny, A. Kempf, and D. W. Kribs},
{\it Generalization of quantum error correction via the Heisenberg picture},
{\rm Phys. Rev. Lett. 98 (2007), 100502}.

\bibitem{bkk1}
{\sc C. B\'{e}ny, A. Kempf, and D. W. Kribs},
{\it Quantum error correction of observables},
{\rm Phys. Rev. A 76 (2007), 042303, 22 pp}.

\bibitem{bkk2}
{\sc C. B\'{e}ny, A. Kempf, and D. W. Kribs},
{\it Quantum error correction on infinite-dimensional Hilbert spaces},
{\rm J. Math. Phys. 50 (2009), no. 6, 062108, 24 pp}.

\bibitem{bfs}
{\sc M. Berta, F. Furrer, and V. B. Scholz},
{\it The smooth entropy formalism for von Neumann algebras},
{\rm J. Math. Phys. 57 (2016), no. 1, 015213, 25 pp}.


\bibitem{blecher_smith}
{\sc D. P. Blecher and R. R. Smith},
{\it The dual of the Haagerup tensor product},
\textrm{J. London Math. Soc. (2) 45 (1992), 126--144}.

\bibitem{cs}
{\sc A. Chatterjee and R. R. Smith},
{\it The central Haagerup tensor product and maps between von Neumann algebras},
{\rm J. Funct. Anal. 112 (1993), no. 1, 97--120}.

\bibitem{clmow}
{\sc E. Chitambar, D. Leung, L. Mancinska, M. Ozols and A. Winter},
{\it Everything you always wanted to know about LOCC (but were afraid to ask)},
\textrm{Comm. Math. Phys. 328 (2014), no. 1, 303--326}.

\bibitem{cklt}
{\sc J. Crann, D.W. Kribs, R.H. Levene and I.G. Todorov},
{\it Private algebras in quantum information and infinite-dimensional complementarity,}
\textrm{J. Math. Phys. 57 (2016), no. 1, 015208}.

\bibitem{d}
{\sc K. R. Davidson},
{\it $C^*$-algebras by Example},
{\rm Fields Institute Monographs, 1996}.

\bibitem{de}
{\sc J. Derezi\'{n}ski},
{\it Introduction to representations of the canonical commutation and anticommutation relations}, in
{\it Large Coulomb Systems: Lecture Notes on Mathematical Aspects of QED},
{\rm Springer Berlin Heidelberg, 2006. pp 63--143}.

\bibitem{dg}
{\sc J. Derezi\'{n}ski and C. G\'{e}rard},
{\it Mathematics of quantization and quantum fields},
{\rm Cambridge Monographs on Mathematical Physics. Cambridge University Press, Cambridge, 2013}.

\bibitem{dmm}
{\sc J. Derezi\'{n}ski, K. A. Meissner, and M. Napi\'{o}rkowski},
{\it On the energy-momentum spectrum of a homogeneous Fermi gas},
{\rm Ann. Henri Poincar\'{e} 14 (2013), no. 1, 1--36}.

\bibitem{DevShor}
{\sc I. Devetak, and P.W. Shor},
{\it The capacity of a quantum channel for simultaneous transmission of classical and quantum information},
{\rm Comm. Math. Phys. 256 (2005), no. 2, 287--303}.


\bibitem{dpp}
{\sc K. Dykema, V. Paulsen, and J. Prakash},
{\it Non-closure of the set of quantum correlations via graphs},
{\rm Comm. Math. Phys. 365 (2019), no. 3, 1125--1142}.


\bibitem{fk}
{\sc T. Fack and H. Kosaki},
{\it Generalised $s$-numbers of $\tau$-measurable operators},
{\rm Pacific J. Math. 123 (1986), no. 2, 269--300}.

\bibitem{Faddeev}
{\sc L. Faddeev},
{\it Discrete Heisenberg-Weyl group and modular group},
{\rm Lett. Math. Phys. 34 (1995), 249}.

\bibitem{haag}
{\sc R. Haag},
{\it Local quantum physics. Fields, particles, algebras},
{\rm Texts and Monographs in Physics. Springer-Verlag, Berlin, 1992}.

\bibitem{haagerup}
{\sc U. Haagerup},
{\it Decomposition of completely bounded maps on operator algebras},
{\rm unpublished manuscript}.

\bibitem{haagerup2}
{\sc U. Haagerup},
{\it The standard form of von Neumann algebras}.
{\rm Math. Scand. 37 (1975), 271--283}.

\bibitem{musath}
{\sc U. Haagerup and M. Musat},
{\it As asymptotic property of factorizable completely positive maps and the Connes embedding problem},
{\rm Comm. Math. Phys. 338 (2015), 721--752}.

\bibitem{h}
{\sc F. Hiai},
{\it Majorization and stochastic maps in von Neumann algebras},
{\rm J. Math. Anal. Appl. 127 (1987), 18--48}.

\bibitem{h2}
{\sc F. Hiai},
{\it Spectral relations and unitary mixing in semifinite von Neumann algebras},
{\rm Hokkaido Math. J. 17 (1988), no. 1, 117--137}.


\bibitem{hiaif1}
{\sc F. Hiai},
{\it Quantum $f$-divergences in von Neumann algebras I. Standard $f$-divergences},
{\rm J. Math. Phys. 59 (2018), no. 10, 102202, 27 pp}.

\bibitem{hiaif2}
{\sc F. Hiai},
{\it Quantum $f$-divergences in von Neumann algebras II. Maximal $f$-divergences},
{\rm J. Math. Phys. 60 (2019), no. 1, 012203, 30 pp}.

\bibitem{hs}
{\sc S. Hollands and K. Sanders},
{\it Entanglement measures and their properties in quantum field theory},
{\rm Springer Briefs in Mathematical Physics 34, 2018. Springer International Publishing}.

\bibitem{jietal}
{\sc Z. Ji, A. Natarajan, T. Vidick, J. Wright and H. Yuen},
{\it MIP$^*$=RE},
{\rm arXiv:2001.04383}.

\bibitem{jungeetal}
{\sc M. Junge, M. Navascues, C. Palazuelos, D. Perez-Garcia, V. B. Scholz, and R. F Werner}.
{\it Connes' embedding problem and Tsirelson's problem},
{\rm J. Math. Phys. 52 (2011), no. 1, 012102}.

\bibitem{jungep}
{\sc M. Junge and C. Palazuelos},
{\it Large violation of Bell inequalities with low entanglement},
{\rm Comm. Math. Phys. 306 (2011), 695--746}.

\bibitem{keylsw}
{\sc M. Keyl, D. Schlingemann and R. F. Werner},
{\it Infinitely entangled states},
{\rm Quantum Inf. Comput. 3 (2003), no. 4, 281--306}.

\bibitem{kmsw}
{\sc M. Keyl, T. Matsui, D. Schlingemann, and R. F. Werner},
{\it Entanglement Haag-duality and type properties of infinite quantum spin chains},
{\rm Rev. Math. Phys. 18 (2006), no. 9, 935--970}.

\bibitem{kr2}
{\sc R. V. Kadison and J. R. Ringrose},
{\it Fundamentals of the theory of operator algebras, Volume II: Advanced theory},
{\rm AMS Grad. Studies in Math. vol.~16, 1997}.


\bibitem{ksw}
{\sc D. Kretschmann, D. Schlingemann and R. F. Werner},
{\it A continuity theorem for Stinespring's dilation},
{\rm J. Funct. Anal. 255 (2008), 1889--1904}.

\bibitem{kuper}
{\sc G. Kuperberg},
{\it The capacity of hybrid quantum memory},
{\rm IEEE Trans. Inf. Thy. 49 (2003), no. 6, 1465--1473}.

\bibitem{lp}
{\sc H.-K. Lo and S. Popescu},
{\it Concentrating entanglement by local actions: Beyond mean values},
{\rm Phys. Rev. A 63 (2001), no. 2, 022301}.

\bibitem{longo}
{\sc R. Longo},
{\it On Landauer's principle and bound for infinite systems},
{\rm Comm. Math. Phys. 363 (2018), no. 2, 531--560}.

\bibitem{lx}
{\sc R. Longo and F. Xu},
{\it Relative entropy in CFT},
{\rm Adv. Math. 337 (2018), 139--170}.

\bibitem{nielsen}
{\sc M. Nielsen},
{\it Conditions for a class of entanglement transformations},
{\rm Phys. Rev. Lett. 83 (1999), 436}.

\bibitem{op}
{\sc M. Ohya and D. Petz},
{\it Quantum entropy and its use},
{\rm Texts and Monographs in Physics, Springer-Verlag, Berlin-Heidelberg 1993}.

\bibitem{obnm}
{\sc M. Owari, S. L. Braunstein, K. Nemoto and M. Murao},
{\it $\ep$-convertibility of entangled states and extension of Schmidt rank in infinite-dimensional systems},
{\rm Quantum Inf. Comput. 8 (2008), no. 1 \& 2, 30--52}.


\bibitem{paulsen_book}
{\sc V. I. Paulsen},
{\it Completely bounded maps and operator algebras},
{\rm Cambridge University Press, 2002}.


\bibitem{psstw}
{\sc V. I. Paulsen, S. Severini, D. Stahlke, I. G. Todorov and A. Winter},
{\it Estimating quantum chromatic numbers},
{\rm J. Funct. Anal. 270 (2016), no. 6, 2188--2222}.


\bibitem{pr}
{\sc S. Popescu and D. Rohrlich},
{\it Thermodynamics and the measure of entanglement},
{\rm Phys. Rev. A (3) 56 (1997), no. 5, R3319--R3321}.


\bibitem{slawny}
{\sc J. Slawny},
{\it On factor representations and the $C^*$-algebra of canonical commutation relations},
{\rm Comm. Math. Phys. 24 (1972), 151--170}.

\bibitem{slofstra}
{\sc W. Slofstra},
{\it Tsirelson's problem and an embedding theorem for groups arising from non-local games},
{\rm J. Amer. Math. Soc. 33 (2020), no. 1, 1--56.}

\bibitem{t1}
{\sc M. Takesaki},
{\it Theory of operator algebras I},
{\rm Encyclopedia of Mathematical Sciences 124, Springer-Verlag Berlin-Heidelberg-New York, 2002}.

\bibitem{t2}
{\sc M. Takesaki},
{\it Theory of operator algebras II},
{\rm Encyclopedia of Mathematical Sciences 125, Springer-Verlag Berlin-Heidelberg-New York, 2003}.

\bibitem{veh}
{\sc T. van Erven and P. Harremo\"{e}s},
{\it R\'{e}nyi Divergence and majorization},
{\rm 2010 IEEE International Symposium on Information Theory, Conference Proceedings DOI:10.1109/ISIT.2010.5513784}.

\bibitem{vw}
{\sc R. Verch and R. F. Werner},
{\it Distillability and positivity of partial transposes in general quantum field systems},
{\rm Rev. Math. Phys. 17 (2005), no. 5, 545--576}.

\bibitem{v}
{\sc G. Vidal},
{\it Entanglement monotones},
{\rm  J. Modern Opt. 47 (2000), no. 2--3, 355--376}.

\bibitem{vN}
{\sc J. von Neumann},
{\it Mathematical foundations of quantum mechanics},
{\rm Princeton University Press, 1955}.

\end{thebibliography}
\end{document}